%% file: definetti-lattice-rev1.tex
\documentclass[11pt,letterpaper]{article}
\usepackage{ifthen,latexsym,amssymb,amsmath,mathdots}
\usepackage{tikz}
\usepackage{multicol}
\usepackage{genyoungtabtikz}
\usepackage{graphicx}
\usepackage{hyperref}

\usepackage{ytableau}

\usetikzlibrary{positioning}

\usepackage{amsthm}
\usepackage{xparse}

\usepackage{chngcntr}
\counterwithin{figure}{section}

\renewcommand{\mid}{:}

\def\R{\mathbb{R}}

\def\cC{\mathcal{C}}

\def\cL{\mathcal{L}}
\def\cK{\mathcal{K}}

\def\cP{\mathcal{P}}

\def\cS{\mathcal{S}}

\newcommand{\oo}[1]{#1^{\circ}}
\newcommand{\maxi}[1]{#1_{\max}}
\newcommand{\mini}[1]{#1_{\min}}

\newcommand{\gog}[1]{\mathcal{G}_{#1}}
\newcommand{\ogog}[1]{\oo{\gog{#1}}}
\newcommand{\magog}[1]{\mathcal{M}_{#1}}
\newcommand{\omagog}[1]{\oo{\magog{#1}}}
\newcommand{\kagog}[1]{\mathcal{K}_{#1}}
\newcommand{\seq}[1]{\mathcal{S}_{#1}}
\newcommand{\coin}[1]{\mathcal{C}_{#1}}
\newcommand{\catpath}[1]{\mathcal{P}_{#1}}
\newcommand{\hybrid}[1]{\mathcal{H}_{#1}}

\newcommand{\pyr}[1]{\triangle{#1}}
\newcommand{\pyro}[1]{\pyr{\oo{#1}}}

\newcommand{\pyrogog}[1]{\pyr{\ogog{#1}}}
\newcommand{\pyrmagog}[1]{\pyr{\magog{#1}}}
\newcommand{\pyromagog}[1]{\pyr{\omagog{#1}}}
\newcommand{\pyrkagog}[1]{\pyr{\kagog{#1}}}
\newcommand{\pyrseq}[1]{\pyr{\seq{#1}'}}
\newcommand{\pyrcoin}[1]{\pyr{\coin{#1}'}}

\DeclareDocumentCommand\finetti{ m g g }{
        \IfNoValueF {#3} {\mathcal{F}_{#1,#2}^{#3}}
        \IfNoValueT {#3} {\mathcal{F}_{#1  \IfNoValueF {#2}{, #2}}}
}

\let\emptyset\varnothing
\let\phi\varphi

\newcommand{\sq}[2]{{\draw (#1,#2) -- (#1+1,#2) -- (#1+1,#2+1) -- (#1,#2+1) -- cycle;}}

\newcommand{\powerset}[1]{{\mathcal{P}(#1)}}

\newif\ifnotesw\noteswtrue

\newcommand{\beq}[1]{\begin{equation}\label{eq:#1}}
\newcommand{\eeq}{\end{equation}}

\newtheorem{theorem}{Theorem}[section]
\newcommand{\bth}[2][nothing]{\ifthenelse{\equal{#1}{nothing}}
 {\begin{theorem}} {\begin{theorem}[#1]}\label{th:#2}}

\newtheorem{lemma}[theorem]{Lemma}
\newcommand{\blm}[2][nothing]{\ifthenelse{\equal{#1}{nothing}}
 {\begin{lemma}} {\begin{lemma}[#1]}\label{lm:#2}}
 
 \newtheorem{prop}[theorem]{Proposition}
\newcommand{\bpr}[2][nothing]{\ifthenelse{\equal{#1}{nothing}}
 {\begin{prop}} {\begin{prop}[#1]}\label{pr:#2}}

\newtheorem{cor}[theorem]{Corollary}
\newcommand{\bcor}[2][nothing]{\ifthenelse{\equal{#1}{nothing}}
 {\begin{cor}} {\begin{cor}[#1]}\label{co:#2}}
 
 \newtheorem{definition}[theorem]{Definition}
\newcommand{\bdef}[2][nothing]{\ifthenelse{\equal{#1}{nothing}}
 {\begin{definition}} {\begin{definition}[#1]}\label{def:#2}}

\newcommand{\frontside}[2]{\fill[xshift=#1 cm, yshift=#2 cm, fill=white,draw=black] (0,0)--(-1cm,0)--(-1cm,1cm)--(0,1cm)--(0,0);}
\newcommand{\topside}[2]{\fill[xshift=#1 cm, yshift=#2 cm, fill=white!98!gray,draw=black] (0,1)--(-1cm,1cm)--(-.5cm,1.3cm)--(0.5cm,1.3cm)--(0,1);}
\newcommand{\rightside}[2]{\fill[xshift=#1 cm, yshift=#2 cm, fill=white!98!gray,draw=black] (0,0)--(0,1cm)--(.5cm,1.3cm)--(0.5cm,0.3cm)--(0,0);}
\newcommand{\justtopside}[2]{\fill[xshift=#1 cm, yshift=#2 cm, fill=white!98!gray,draw=black] (0,0)--(-1cm,0cm)--(-.5cm,0.3cm)--(0.5cm,0.3cm)--(0,0);}
\newcommand{\justrightside}[2]{\fill[xshift=#1 cm, yshift=#2 cm, fill=white!98!gray,draw=black] (0,0)--(0.5cm,0.3cm)--(0.5,0cm)--(0,0);}

\newcommand{\gfrontside}[2]{\fill[xshift=#1 cm, yshift=#2 cm, fill=gray!30!white,draw=black] (0,0)--(-1cm,0)--(-1cm,1cm)--(0,1cm)--(0,0);}
\newcommand{\gtopside}[2]{\fill[xshift=#1 cm, yshift=#2 cm, fill=gray!45!white,draw=black] (0,1)--(-1cm,1cm)--(-.5cm,1.3cm)--(0.5cm,1.3cm)--(0,1);}
\newcommand{\grightside}[2]{\fill[xshift=#1 cm, yshift=#2 cm, fill=gray!45!white,draw=black] (0,0)--(0,1cm)--(.5cm,1.3cm)--(0.5cm,0.3cm)--(0,0);}
\newcommand{\gjusttopside}[2]{\fill[xshift=#1 cm, yshift=#2 cm, fill=gray!45!white,draw=black] (0,0)--(-1cm,0cm)--(-.5cm,0.3cm)--(0.5cm,0.3cm)--(0,0);}
\newcommand{\gjustrightside}[2]{\fill[xshift=#1 cm, yshift=#2 cm, fill=gray!45!white,draw=black] (0,0)--(0.5cm,0.3cm)--(0.5,0cm)--(0,0);}

\newcommand{\cube}[2]{
\frontside{#1}{#2}
\topside{#1}{#2}
\rightside{#1}{#2}}

\newcommand{\gcube}[2]{
\gfrontside{#1}{#2}
\gtopside{#1}{#2}
\grightside{#1}{#2}}

\newcounter{x}
\newcounter{r}

\newcommand{\blocklayer}[2]{
\foreach \i in {1,...,#1}{
	\cube{\i}{0}}
\setcounter{x}{-1}
\foreach \j/\off in {#2}{
	\addtocounter{x}{1}
	\foreach \k in {1,...,\j}{
		\pgfmathsetmacro\y{\k+\off+0.5*(\value{x}+1)}
		\pgfmathsetmacro\z{1.3+\value{x}*0.3}
		\justtopside{\y}{\z}}
	\pgfmathsetmacro\Y{\j+\off+0.5*(\value{x}+1)}
	\pgfmathsetmacro\Z{1.3+\value{x}*0.3}	
	\justrightside{\Y}{\Z}
}
}

\newcommand{\twocolor}[2]{
\setcounter{x}{-1}
\foreach \g/\w in {#1}{
	\ifthenelse{\g=0}{}{
	\foreach \i in {1,...,\g}{
		\addtocounter{x}{1}
		\pgfmathsetmacro\X{\value{x}}
		\gcube{\X}{0}}}
	\ifthenelse{\w=0}{}{
	\foreach \i in {1,...,\w}{	
		\addtocounter{x}{1}
		\pgfmathsetmacro\X{\value{x}}
		\cube{\X}{0}}}}
\setcounter{x}{-1}
\foreach \g/\w in {#2}{
	\addtocounter{x}{1}
	\setcounter{r}{-1}
	\ifthenelse{\g=0}{}{
	\foreach \k in {1,...,\g}{
		\addtocounter{r}{1}
		\pgfmathsetmacro\y{\value{r}+0.5*(\value{x}+1)}
		\pgfmathsetmacro\z{1.3+\value{x}*0.3}
		\gjusttopside{\y}{\z}}}
	\ifthenelse{\w=0}{\pgfmathsetmacro\Y{\value{r}+0.5*(\value{x}+1)}
	\pgfmathsetmacro\Z{1.3+\value{x}*0.3}
	\gjustrightside{\Y}{\Z}}{
	\foreach \k in {1,...,\w}{
		\addtocounter{r}{1}
		\pgfmathsetmacro\y{\value{r}+0.5*(\value{x}+1)}
		\pgfmathsetmacro\z{1.3+\value{x}*0.3}
		\justtopside{\y}{\z}}	
	\pgfmathsetmacro\Y{\value{r}+0.5*(\value{x}+1)}
	\pgfmathsetmacro\Z{1.3+\value{x}*0.3}
	\justrightside{\Y}{\Z}}
}
}

\newcommand{\reversetwocolor}[2]{
\setcounter{x}{-1}
\foreach \g/\w in {#1}{
	\ifthenelse{\g=0}{}{
	\foreach \i in {1,...,\g}{
		\addtocounter{x}{1}
		\pgfmathsetmacro\X{\value{x}}
		\cube{\X}{0}}}
	\ifthenelse{\w=0}{}{
	\foreach \i in {1,...,\w}{	
		\addtocounter{x}{1}
		\pgfmathsetmacro\X{\value{x}}
		\gcube{\X}{0}}}}
\setcounter{x}{-1}
\foreach \g/\w in {#2}{
	\addtocounter{x}{1}
	\setcounter{r}{-1}
	\ifthenelse{\g=0}{}{
	\foreach \k in {1,...,\g}{
		\addtocounter{r}{1}
		\pgfmathsetmacro\y{\value{r}+0.5*(\value{x}+1)}
		\pgfmathsetmacro\z{1.3+\value{x}*0.3}
		\justtopside{\y}{\z}}}
	\ifthenelse{\w=0}{\pgfmathsetmacro\Y{\value{r}+0.5*(\value{x}+1)}
	\pgfmathsetmacro\Z{1.3+\value{x}*0.3}
	\justrightside{\Y}{\Z}}{
	\foreach \k in {1,...,\w}{
		\addtocounter{r}{1}
		\pgfmathsetmacro\y{\value{r}+0.5*(\value{x}+1)}
		\pgfmathsetmacro\z{1.3+\value{x}*0.3}
		\gjusttopside{\y}{\z}}	
	\pgfmathsetmacro\Y{\value{r}+0.5*(\value{x}+1)}
	\pgfmathsetmacro\Z{1.3+\value{x}*0.3}
	\gjustrightside{\Y}{\Z}}
}
}

\begin{document}

\title{de Finetti Lattices and Magog Triangles} 
\author{Andrew Beveridge\footnote{Department  of Mathematics, Statistics and Computer Science, Macalester College, St Paul, MN 55105. \texttt{abeverid@macalester.edu}},\; Ian Calaway\footnote{Department of Economics, Stanford University, Stanford, CA 94305. \texttt{icalaway@stanford.edu}}
\; and Kristin Heysse\footnote{Department  of Mathematics, Statistics and Computer Science, Macalester College, St Paul, MN 55105. \texttt{kheysse@macalester.edu}}}
\maketitle

\begin{abstract}
The order ideal $B_{n,2}$ of the Boolean lattice $B_n$ consists of all subsets of size at most $2$. Let $F_{n,2}$ denote the poset refinement of $B_{n,2}$ induced by the rules: $i < j$ implies 
$\{i \} \prec \{ j \}$ and $\{i,k \} \prec \{j,k\}$. We give an elementary bijection from
the set $\finetti{n}{2}$ of linear extensions of $F_{n,2}$ to the set of  shifted standard Young tableau of shape $(n, n-1, \ldots, 1)$, which are counted by the strict-sense ballot numbers. We find a more surprising result when considering the set $\finetti{n}{2}{1}$  of minimal poset refinements in which each singleton is comparable with all of the doubletons. We show that $\finetti{n}{2}{1}$ is in bijection with magog triangles, and therefore is equinumerous with alternating sign matrices. 
We adopt our proof techniques to show that row reversal of an alternating sign matrix corresponds to a natural involution on gog triangles.
\end{abstract}



    


\input{newintro}

\input{background}

\input{finetti-ballot.tex}

\input{kagog-intro.tex}

\input{finetti-kagog}

\input{kagog-magog.tex}

\input{gog-results.tex}

\input{future.tex}
{\bf Acknowledgments.} We thank David Bressoud and Jessica Striker for sharing their insights about ASMs, TSSCPPs, triangular arrays and posets. 

\bibliography{refs} 

\end{document}

%% file: newintro.tex
\section{Introduction}

\subsection{Overview}
\label{sec:newintro}

The Boolean lattice $B_n$ consists of subsets of $[n] = \{ 1,2, \ldots, n \}$ ordered by inclusion. 
The  order ideal $B_{n,2} \subset B_n$ consists of all subsets of size at most $2$.
Define $F_{n,2}$ to be the poset refinement of $B_{n,2}$ where we add the relations
\begin{enumerate}
\item[(R1)] $\{i \} \prec \{ k \}$ if and only if   $i < k$.
\item[(R2)] If $i<j$ and $k < \ell$ then 
$$
\{ i, j \} \prec \{ k, \ell \} \quad \mbox{if and only if} \quad ( i < k \mbox{ and } j \leq \ell) \mbox{ or } 
( i \leq k \mbox{ and } j < \ell).
$$
\end{enumerate}
Transitivity of the set inclusion relation $\{i\} \prec \{i,j\}$ and (R2) yields the relation
$$
\{ i \} \prec \{ k, \ell \} \quad \mbox{if and only if} \quad   i < \min \{ k,\ell \}.
$$
The Hasse diagrams for $B_{n,2}$ and $F_{n,2}$ are shown in Figure \ref{fig:boolean-finetti}.

\begin{figure}[ht]

\begin{tikzpicture}[align=center,node distance=1.25cm and 1.25cm]

\begin{scope}[shift={(0,1.5)}]
\node at (0,3.25) {$B_{4,2}$};
\node (b0) at (0,0) {$\emptyset$};;

\node[above left of=b0] (b2)  {2};
\node[left of=b2]  (b1)  {1};
\node[right of=b2] (b3)  {3};
\node[right of=b3] (b4)  {4};

\node[above of=b1] (b31) {31};
\node[left of=b31] (b21) {21};
\node[right of=b31] (b41) {41};
\node[right of=b41] (b32) {32};
\node[right of=b32] (b42) {42};
\node[right of=b42] (b43) {43};

\draw[thick] (b0) -- (b1);
\draw[thick] (b0) -- (b2);
\draw[thick] (b0) -- (b3);
\draw[thick] (b0) -- (b4);

\draw[thick] (b1) -- (b21);
\draw[thick] (b1) -- (b31);
\draw[thick] (b1) -- (b41);
\draw[thick] (b2) -- (b21);
\draw[thick] (b2) -- (b32);
\draw[thick] (b2) -- (b42);
\draw[thick] (b3) -- (b31);
\draw[thick] (b3) -- (b32);
\draw[thick] (b3) -- (b42);
\draw[thick] (b4) -- (b41);
\draw[thick] (b4) -- (b42);
\draw[thick] (b4) -- (b43);
\end{scope}

\begin{scope}[shift={(8,0)}]
\node at (-1.5,5.5) {$F_{4,2}$};
\node (n0) at (0,0) {$\emptyset$};;
\node[above right of=n0]  (n1)  {1};
\node[above left of=n1] (n2)  {2};
\node[above left of=n2] (n3)  {3};
\node[above left of=n3] (n4)  {4};
\node[above right of=n2] (n21) {21};
\node[above right of=n3] (n31) {31};
\node[above right of=n31] (n32) {32};
\node[above right of=n4] (n41) {41};
\node[above right of=n41] (n42) {42};
\node[above  right of=n42] (n43) {43};
\draw[thick] (n4) -- (n3) -- (n2) -- (n1) -- (n0);
\draw[thick] (n2) -- (n21) -- (n31) -- (n41) -- (n42);
\draw[thick]  (n31) --  (n32) -- (n42) -- (n43);
\draw[thick] (n3) -- (n31);
\draw[thick] (n4) -- (n41);
\end{scope}

\end{tikzpicture}

\caption{Hasse diagrams for the order ideal $B_{n,2}$ and the de Finetti lattice $F_{n,2}$. The sets $\{i\}$ and $\{j,k\}$ are denoted by $i$ and $jk$ where $j > k$.}
\label{fig:boolean-finetti}
\end{figure}
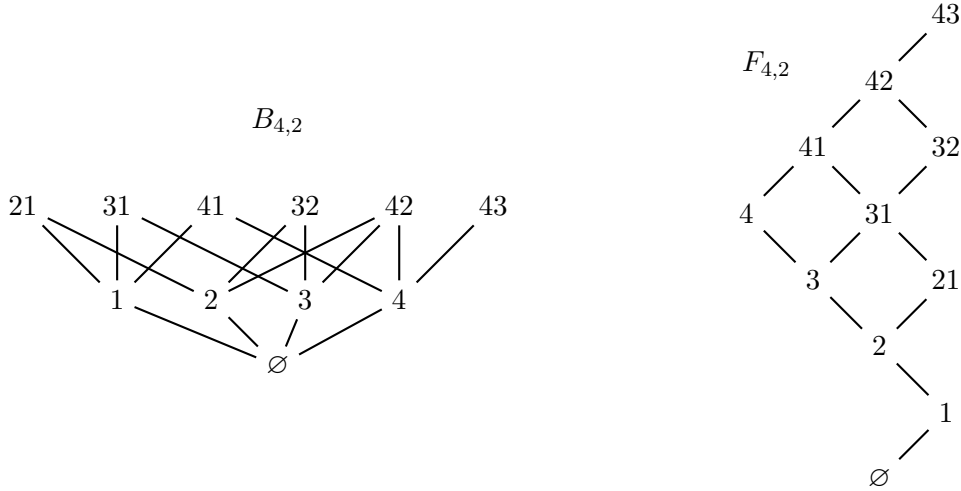


It is straightforward to confirm that the poset $F_{n,2}$ is a distributive lattice. First, we show every pair of subsets $x,y \in F_{n,2}$ has
a meet (greatest lower bound) $x \wedge y$ and a join (least upper bound) $x \vee y$.
For $j > i$  and $ \ell > k$, we have
$$
\begin{array}{rclcrcl}
\{ j  \} \wedge \{ \ell \} &=& \{ \min \{ j,\ell \} \},& \quad & \{ j  \} \vee \{ \ell \} &=& \{ \max \{ j,\ell \}  \}, \\
\{ j  \} \wedge \{ \ell,k \} &=& \{ \min \{ j,\ell \} \},& \quad & \{ j  \} \vee \{ \ell, k \} &=& \{ \max \{ j,\ell \}, k \}, \\
\{ j,i  \} \wedge \{\ell,k \} &=& \{ \min \{ j,\ell \}, \min \{ i,k \} \},& \quad & \{ j,i  \} \vee \{ \ell, k \} &=& \{ \max \{ j,\ell \}, \max \{i,k\} \}.\\
\end{array}
$$
Note that if we identify $\{j \}$ with $\{j,0\}$ and $\{\ell \}$ with $\{\ell,0\}$, then the meet and join rules for doubleton pairs imply the other rules. Moreover, it then  becomes easy to confirm that $F_{n,2}$ is a distributive lattice: $x \vee (y \wedge z) =  (x \vee y) \wedge (x \vee z)$ and $x \wedge (y \vee z) = (x \wedge y) \vee (x \wedge z)$.

We refer to $F_{n,2}$ as the \emph{de Finetti lattice}; the origin of this name will be illuminated in Section \ref{sec:csp} below. We resolve two questions concerning families of poset refinements of $F_{n,2}$.

First, let $\finetti{n}{2}$ denote the collection of linear extensions of $F_{n,2}$. We give a simple bijection between the total orders in $\finetti{n}{2}$ and shifted standard Young tableau (shifted SYT) of shape $(n,n-1,\ldots,1)$, see OEIS A003121 \cite{oeis}. In these shifted SYT of staircase shape, the first box in row $i > 1$ is located below the second box of row $i-1$. The integers $1,2, \ldots ,n(n+1)/2$ are arranged in the boxes so that the  rows and the columns are both increasing. These are equinumerous with the number of strict-sense ballots with $n$ candidates, where candidate $k$ gets $k$ votes, candiate $k$ never trails candidate $\ell$ for $n \geq k > \ell \geq  1$, see 
\cite{barton}. For $1 \leq n \leq 7$, the strict-sense ballot numbers are
$$
1, 1, 2, 12, 286, 33592, 23178480
$$
and the general formula for the $n$th  strict-sense ballot number is
$$
\binom{n+1}{2} ! \, \frac{\prod_{k=1}^{n-1} k!}{\prod_{k=1}^n (2k-1)!}.
$$
The proof of the following proposition appears in Section \ref{sec:finetti-ballot}.

\begin{prop}
\label{thm:fn2}
The set $\finetti{n}{2}$ is in bijection with shifted standard Young tableaux of shape $(n,n-1,\ldots,1)$. Therefore $\finetti{n}{2}$ is enumerated by the strict-sense ballot numbers.
\end{prop}

Our second family of poset refinements is less conventional. Rather than performing a linear extension of $B_{n,2}$, we add only the relations required so that the singleton sets are  comparable with every doubleton set. In other words, we resolve every incomparable pair $\{i\}, \{k,\ell\}$, but  this may leave some incomparable doubleton pairs 
$\{i,j\}, \{k,\ell\}$.
Let $\finetti{n}{2}{1}$ denote this family of  poset refinements of $F_{n,2}$. We prove the following lemma in Section \ref{sec:csp}.
\begin{lemma}  
\label{lemma:Fn21}
Every poset in $\mathcal{F}_{n,2}^{1}$ is a lattice.
\end{lemma}
The seven posets in $\finetti{4}{2}{1}$ are shown in Figure \ref{fig:finetti-4-2-1}.
Two of these posets are linear extensions. In the remaining five refinements, the 
 doubletons 41 and 32 are incomparable. There are two ways to resolve this relation: either $41 \prec 32$ or $32 \prec 41$. This accounts for the  12 linear extensions in $\finetti{n}{2}$, in accordance with Proposition \ref{thm:fn2}.


\begin{figure}[ht]
\begin{center}
\begin{tikzpicture}[align=center,node distance=.9cm] 

\begin{scope}

\node (n0) at (1,0) {$\emptyset$};;
\node[above of=n0]  (n1)  {1};
\node[above of=n1] (n2)  {2};
\node[above of=n2] (n21) {21};

\node[above of=n21] (n3)  {3};
\node[above of=n3] (n31) {31};
\node[above of=n31] (n32) {32};

\node[above of=n32] (n4)  {4};
\node[above of=n4] (n41) {41};
\node[above of=n41] (n42) {42};
\node[above  of=n42] (n43) {43};

\draw (n0) -- (n1) -- (n2) -- (n21) -- (n3) -- (n31) -- (n32);
\draw (n32) -- (n4) -- (n41) -- (n42) -- (n43);

\end{scope}

\begin{scope}[shift={(2.5,0)}]

\node (n0) at (0,0) {$\emptyset$};;
\node[above of=n0]  (n1)  {1};
\node[above of=n1] (n2)  {2};

\node[above of=n2] (n3)  {3};
\node[above of=n3] (n21) {21};

\node[above of=n21] (n31) {31};
\node[above of=n31] (n32) {32};

\node[above of=n32] (n4)  {4};
\node[above of=n4] (n41) {41};
\node[above of=n41] (n42) {42};
\node[above  of=n42] (n43) {43};

\draw (n0) -- (n1) -- (n2) -- (n3)-- (n21)  -- (n31) -- (n32);
\draw(n32) -- (n4) -- (n41) -- (n42) -- (n43);

\end{scope}

\begin{scope}[shift={(4.5,0)}]

\node (n0) at (0,0) {$\emptyset$};;
\node[above of=n0]  (n1)  {1};
\node[above of=n1] (n2)  {2};

\node[above of=n2] (n21) {21};
\node[above of=n21] (n3)  {3};

\node[above of=n3] (n31) {31};
\node[above of=n31] (n4)  {4};

\node[above left of=n4] (n32) {32};
\node[above right of=n4] (n41) {41};

\node[above left of=n41] (n42) {42};
\node[above  of=n42] (n43) {43};

\draw (n0) -- (n1) -- (n2) -- (n21) -- (n3)  -- (n31) -- (n4) -- (n32);
\draw (n4) -- (n41) -- (n42) -- (n43);
\draw (n32) -- (n42);

\end{scope}

\begin{scope}[shift={(7,0)}]

\node (n0) at (0,0) {$\emptyset$};;
\node[above of=n0]  (n1)  {1};
\node[above of=n1] (n2)  {2};

\node[above of=n2] (n3)  {3};
\node[above of=n3] (n21) {21};

\node[above of=n21] (n31) {31};
\node[above of=n31] (n4)  {4};

\node[above left of=n4] (n32) {32};
\node[above right of=n4] (n41) {41};

\node[above left of=n41] (n42) {42};
\node[above  of=n42] (n43) {43};

\draw (n0) -- (n1) -- (n2) -- (n3)-- (n21)  -- (n31) -- (n4) -- (n32);
\draw (n4) -- (n41) -- (n42) -- (n43);
\draw (n32) -- (n42);

\end{scope}

\begin{scope}[shift={(9.5,0)}]

\node (n0) at (0,0) {$\emptyset$};;
\node[above of=n0]  (n1)  {1};
\node[above of=n1] (n2)  {2};

\node[above of=n2] (n21) {21};
\node[above of=n21] (n3)  {3};

\node[above of=n3] (n4)  {4};
\node[above of=n4] (n31) {31};

\node[above left of=n31] (n32) {32};
\node[above right of=n31] (n41) {41};

\node[above left of=n41] (n42) {42};
\node[above  of=n42] (n43) {43};

\draw (n0) -- (n1) -- (n2) -- (n21) -- (n3)  -- (n4) -- (n31)  -- (n32);
\draw (n31) -- (n41) -- (n42) -- (n43);
\draw (n32) -- (n42);

\end{scope}

\begin{scope}[shift={(12,0)}]

\node (n0) at (0,0) {$\emptyset$};;
\node[above of=n0]  (n1)  {1};
\node[above of=n1] (n2)  {2};

\node[above of=n2] (n3)  {3};
\node[above of=n3] (n21) {21};

\node[above of=n21] (n4)  {4};
\node[above of=n4] (n31) {31};

\node[above left of=n31] (n32) {32};
\node[above right of=n31] (n41) {41};

\node[above left of=n41] (n42) {42};
\node[above  of=n42] (n43) {43};

\draw (n0) -- (n1) -- (n2) -- (n3)-- (n21) -- (n4) -- (n31)  -- (n32);
\draw (n31) -- (n41) -- (n42) -- (n43);
\draw (n32) -- (n42);

\end{scope}

\begin{scope}[shift={(14.5,0)}]

\node (n0) at (0,0) {$\emptyset$};;
\node[above of=n0]  (n1)  {1};
\node[above of=n1] (n2)  {2};

\node[above of=n2] (n3)  {3};

\node[above of=n3] (n4)  {4};

\node[above of=n4] (n21) {21};

\node[above of=n21] (n31) {31};

\node[above left of=n31] (n32) {32};
\node[above right of=n31] (n41) {41};

\node[above left of=n41] (n42) {42};
\node[above  of=n42] (n43) {43};

\draw (n0) -- (n1) -- (n2) -- (n3) -- (n4) -- (n21) -- (n31)  -- (n32);
\draw (n31) -- (n41) -- (n42) -- (n43);
\draw (n32) -- (n42);

\end{scope}

\end{tikzpicture}

\caption{Hasse diagrams for the seven posets in $\finetti{4}{2}{1}$.}
\label{fig:finetti-4-2-1}

\end{center}
\end{figure}

Our main result is a bijection between $\finetti{n}{2}{1}$ and \emph{magog triangles} of size $n-1$, see OEIS A005130 \cite{oeis}.
Magog triangles are a family of triangular integer arrays which are in bijection with totally symmetric self-complimentary plane partitions (TSSCPPs). Their counterpart \emph{gog triangles} are in bijection with alternating sign matrices (ASMs). Gog and magog triangles were
 instrumental to Zielberger's famous proof that ASMs are equinumerous with TSSCPPs  \cite{zielberger}. 
The first seven numbers in this sequence are
$$
1, 2, 7, 42, 429, 7436, 218348
$$
and the general formula is
\begin{equation}
\label{eqn:asm}
\prod_{k=0}^{n-1} \frac{(3k+1)!}{(n+k)!}.
\end{equation}
We use $\magog{n}$ and $\gog{n}$ to denote the respective families of magog triangles and gog triangles of size $n$. 
Here is our main theorem, which we prove in Section \ref{sec:kagog}.
\begin{theorem}
\label{thm:finetti-magog}
The family $\finetti{n}{2}{1}$ of de Finetti refinements of the lattice  $F_{n,2}$ is in bijection with the family $\magog{n-1}$ of magog triangles of size $n-1$. 
\end{theorem}

We make two remarks about our proof of Theorem \ref{thm:finetti-magog}. First, our bijection makes use of a new triangular family that is in bijection with magog triangles, though we defer the description of these \emph{kagog triangles} to Section \ref{sec:kagog-overview}. Second, we view these triangular arrays of positive integers as pyramids of cubes colored gray and white, adhering to appropriate stacking rules. This geometric viewpoint is essential to our proof, which employs an affine transformation and a color inversion to turn a kagog pyramid into a magog pyramid. 

This two-color cube pyramid model may be useful to others interested in studying the enigmatic relationship between ASMs and TSSCPPs. As an example of its potential utility, we  prove the following theorem in Section \ref{sec:gog}.

\begin{theorem}
\label{thm:gog}
Reversing the order of the rows of a $n \times n$ ASMs induces an involution on gog triangles $\gog{n}$. The corresponding involution on two-color cube pyramids reverses the  cube coloring and performs a rigid transformation of the  pyramid.
\end{theorem}

\subsection{de Finetti Lattices}
\label{sec:csp}

In this section, we motivate the study of $F_{n,2}$ and its poset refinements. We also prove Lemma \ref{lemma:Fn21}
In various settings, (including probability, computational algebra, and social choice theory) researchers have investigated  total orders of the power set $\mathcal{P}([n])$ 
satisfying the following two conditions:
\begin{enumerate}
\item[(F1)] $\emptyset \prec \{1\} \prec \{2\} \prec \ldots \prec \{n\}$, and 
\item[(F2)] $X \prec Y$ if and only if $X \cup Z \prec Y \cup Z$
for all $Z \subset [n] \mbox{ such that }  (X \cup Y) \cap Z = \emptyset$.
\end{enumerate}
Condition (F1) is the canonical ordering of the singleton sets. Condition (F2) is  de Finetti's axiom \cite{definetti}, which was first formulated in a probabilistic setting.
This axiom can be restated as $A \prec B$  if and only  if $A \backslash B \prec B \backslash A$. 
Intuitively, condidtion (F2) states that when we have comparable sets, adding the same element (or elements) to both sets will not change the comparison. 
For  $n \geq 3$, conditions  (F1) and (F2) do not completely determine a total order; for example, we cannot deduce whether $\{1,2\} \prec \{ 3\}$ or 
$\{3\} \prec \{ 1,2\}$ from these first principles. 

We make two observations about total orders satisfying these conditions. First, induction  on $|B \backslash A|$ shows that if $A \subsetneq B$ then $A \prec B$.  Second, another simple induction proof confirms that  if $x_i \leq y_i$ for $1 \leq i \leq k$ then (F1) and (F2) lead to the (intuitive) conclusion that 
$\{ x_1, x_2, \ldots, x_k \} \preceq \{ y_1, y_2, \ldots, y_k \}$, where equality holds only when these sets are identical. In summary, we have
a linear extension of the Boolean lattice that also extends the standard ordering on the integers $[n]$ to an ordering of the subsets of $[n]$.

Total orders of $\mathcal{P}([n])$ satisfying (F1) and (F2) appear under various names, including comparative probability orders, Boolean term orders, and completely separable preference orders, see OEIS A005806 \cite{oeis}. 
Each of these names reflects the application setting rather than the defining properties of the total order. Therefore, we opt for the generic name \emph{de Finetti total order},  paying homage to de Finetti's axiom. Furthermore, there is no harm in starting with the Boolean lattice $B_n$ rather than the set $\mathcal{P}([n])$, since the set inclusion relations are enforced by (F1) and (F2).

\begin{definition}
\label{def:finetti-order}
A \emph{de Finetti refinement} $(E, \prec_E)$ of the Boolean lattice $(B_{n}, \prec)$ is a poset refinement that adheres to (F1) and  to (F2) for all sets $X,Y \subset [n]$ that are comparable in $E$.
A \emph{de Finetti total order} is a linear extension of $B_n$ that adheres to (F1) and (F2). 
The collection of de Finetti total orders of $B_n$ is denoted $\finetti{n}$.
\end{definition}

Note that (F2) does not require that all pairs $X$ and $Y$ are comparable, but when they are, the sets $X \cup Z$ and $Y \cup Z$ are also comparable, as are $X \backslash Y$ and $Y \backslash X$.
The number $|\finetti{n}|$ of de Finetti total orders for $1 \leq n \leq 7$ is
$$
1, 1, 2, 14, 546, 169444, 560043206
$$
but there is still no known general formula. 
Enumerations of the 14 de Finetti total orders for $n=4$ can  be found in \cite{fishburn2002,bradley, christian}.
This current work germinated while studying de Finetti total orders: we restricted our attention to the  order ideal $B_{n,2} \subset B_n$ of subsets of size at most $2$, and then considered the poset refinements of $B_{n,2}$ that adhere to (F1) and (F2). 
\begin{definition}
\label{def:finetti2}
A \emph{de Finetti refinement} $(E_2, \prec_{E_2})$ of the order ideal $(B_{n,2}, \prec)$ is a poset refinement that adheres to (F1) and  to (F2) for all sets $X,Y$ that are comparable in $E_2$.
The collection of de Finetti total orders of $B_{n,2}$ is denoted $\finetti{n}{2}$.
\end{definition}

When restricting ourselves to $B_{n,2}$, the conditions (F1) and (F2) simplify to the set inclusion relations (as noted above) plus the relations (R1) and (R2).
Indeed, (F1) clearly implies (R1). As for (R2), let $i < j$ and $k < \ell$. If 
$i < k$ and $j \leq \ell$ then $\{i, j\} \prec \{k, j\} \preceq \{ k, \ell \}$, so $\{i, j\} \prec \{ k, \ell \}$ by transitivity. The case $i \leq k$ and $j < \ell$ proceeds similarly.

Having made this connection, we illuminate how the results in Section \ref{sec:newintro} relate back to de Finetti refinements.
First, we now recognize the  lattice $F_{n,2}$ 
as the unique minimal de Finetti refinement of $B_{n,2}$. Indeed, any de Finetti refinement of $B_{n,2}$ must adhere to conditions (R1) and (R2), so it must contain the lattice $F_{n,2}$. 
Second, any valid poset refinement of $F_{n,2}$ is automatically a de Finetti refinement. The incomparable pairs $X,Y$ of $F_{n,2}$ are disjoint and at least one must a doubleton set, so $Z=\emptyset$ is the only allowed choice in (F2). In other words, defining $\finetti{n}{2}$ as the collection of de Finetti total orders of $B_{n,2}$ (as in Definition \ref{def:finetti2}) is equivalent to our definition of $\finetti{n}{2}$ as the collection of linear extensions of $F_{n,2}$ (as in Section \ref{sec:newintro}).

We can make a similar observation about our second family of poset refinements. 
 Recall that $\finetti{n}{2}{1}$ denotes the collection of minimal poset refinements of the lattice $F_{n,2}$ in which the singleton sets are comparable with all other elements. This is equivalent to the following definition.
 
 \begin{definition}
\label{def:finetti2-1}
The set $\finetti{n}{2}{1}$ is the collection
of minimial de Finetti refinements of $B_{n,2}$ such that the singleton sets are comparable with every set. 
\end{definition}

We now prove Lemma \ref{lemma:Fn21}: each poset refinement in $\finetti{n}{2}{1}$ is a lattice.

\begin{proof}[Proof of Lemma \ref{lemma:Fn21}]
 Consider any pair of elements $x,y \in F_{n,2}$. Since $F_{n,2}$ is a  lattice, these elements have  least upper  bound $x\vee y$, and  greatest lower bound $x\wedge y$ in $F_{n,2}$. 
Let $(P, \prec_P)  \in \mathcal{F}_{n,2}^1$. Since $(P,\prec_P)$ is a refinement of  $(F_{n,2}, \prec)$, we know that $x\vee y$ is still an upper bound of $x$ and $y$ in $P$ and
that that $x\wedge y$ is still  a lower bound of $x$ and $y$ in $P$.

We must show $x,y$ have  a least upper bound $x \vee_P y$ and a greatest lower bound $x \wedge_P y$ in $P$. If at least one of $x,y$ is a singleton set, then $x$ and $y$ are comparable in $P$, which means that $x \vee_P y = \max \{ x,y \}$ and 
$x \wedge_P y = \min \{ x,y \}$. 
So we now assume that  both $x$ and $y$ are doubleton sets.
 
Let us prove that $x$ and $y$ have a   least upper bound in $P$.  
Assume for the sake of contradiction that 
$z \in P$ is a minimal $x,y$ upper bound that is incomparable with $x \vee y$ in $P$.
Recall that the empty set and the singleton sets are comparable with every element in $P$, so 
both $z$ and $x \vee y$ must be doubleton sets. Furthermore,
$z$ is an upper bound of at most one of $x$ and $y$ in $F_{n,2}$: otherwise 
$x \vee y \prec z$ in $F_{n,2}$ and therefore $x \vee y \prec_P z$.
 There are two cases.

First, suppose that $z$ is $F_{n,2}$-incomparable with both $x$ and $y$. In order to be comparable in $P$, there must be singletons $s_1,s_2$ (not necessarily distinct) such that  $x\prec_P s_1 \prec_P z$ and $y\prec_P s_2 \prec_P z$ in $P$. Without loss of generality, $s_1 \leq s_2$, which means that $s_2$ is an $x,y$ upper bound in $P$ and  $ s_2 \prec_P z$. This contradicts the minimality of $z$.

Second, suppose (without loss  of generality) that $x \prec z$ in $F_{n,2}$ while $y$ and $z$ are $F_{n,2}$-incomparable. There must be a singleton $s$ such that
$y \prec_P s \prec_P z$. Now, if $x \prec_P s$ then $s$ is an $x,y$ upper bound in $P$, contradicting the minimality of $z$. Otherwise, $s \prec_P x$ which means that $x = x \vee_P y$, which also contradicts the minimality of $z$.

The doubletons $x,y$ also have a greatest lower bound in $P$: this proof is entirely parallel  to the least upper bound proof. 
%
%
%
%
%
\end{proof}

This concludes our motivation for studying $\finetti{n}{2}$ and $\finetti{n}{2}{1}$. Given our success at characterizing these refinements of $B_{n,2}$, it is natural to formulate analogous research questions for de Finetti refinements of $B_{n,m}$. We defer those formulations to our concluding section.

%% file: background.tex
\subsection{Related Work}

\subsubsection{The Boolean Lattice}

A \emph{partially ordered set} (or \emph{poset} for short) consists of a set $P$ and a binary relation $\preceq$ that is reflexive ($x \preceq x$), antisymmetric (if $x \preceq y$ and $y \preceq x$ then $x=y$) and transitive (if $x \preceq y$ and $y \preceq z$ then $x \preceq z$). 
A \emph{lattice} is a poset such that every pair of elements have a least upper bound and a greatest lower bound.
We obtain a \emph{refinement} of a poset $P$ by adding relations between pairs of incomparable elements of $P$.
A \emph{total order} is a poset in which every pair of elements is comparable. 
A \emph{linear extension} of a poset $P$ is a refinement of $P$ that is a total order.
See Chapter 3 of Stanley \cite{stanley} for an introduction to posets and lattices.


For a poset $P$, let $\cL(P)$ denote the set of linear extensions of $P$.
Brightwell and Tetali \cite{brightwell} determined an accurate asymptotic formula for
$| \mathcal{L}(B_n)|$, improving on work of Sha and Kleitman \cite{sha}. The value of 
$|\mathcal{L}(B_n)|$ is known for $1 \leq n \leq 7$, see OEIS A046873 \cite{oeis}. The $n=7$ case was recently determined by Brower and Christensen \cite{brouwer+christensen} using machinery developed to study the game of Chomp played on the Boolean lattice. 
Pruesse and Ruskey \cite{pruesse} introduced the linear extension graph $G(B_n)$ whose  vertex set is $\mathcal{L}(B_n)$ and whose edge set consists of pairs of linear extensions that differ by a single adjacent transposition. Felsner and Massow \cite{felsner+massow} determined the diameter of 
$G(B_n)$.

Researchers have also studied linear extensions of subposets of $B_n$, including the order ideal $B_{n,m}$ of subsets of size at most $m$. Fink and Gregor \cite{fink+gregor} determined the linear extension diameter of the subposet $B_n^{1,k}$ of $B_n$ that is induced by levels 1 and $k$.  Brouwer and Christensen \cite{brouwer+christensen} determined that 
$$
| \cL(B_{n,2}) | = \frac{n! \left( {n \choose 2} + n \right)!}{ \prod_{i=1}^n \left( in - {i \choose 2} \right)}
={n+1 \choose 2}!  \frac{1}{\prod_{i=1}^n \left( n - \frac{i-1}{2} \right)}
$$
and computed $| \cL(B_{n,3}) |$ for $n \leq 7$. Comparing this formula with our Proposition \ref{thm:fn2} shows that $n! \cdot | \finetti{n,2}| = o(|\cL(B_{n,2})|)$. In other words,
 the de Finetti linear extensions of $B_{n,2}$ are exceptionally rare.

\subsubsection{de Finetti  Total Orders}

The de Finetti total orders $\finetti{n}$ are total orders of $\powerset{[n]}$ that satisfy both (F1) and (F2). 
These total orders appear in a variety of settings with names that reflect the application at hand \cite{fine,maclagan,bradley,christian}. 
In probability theory, the total orders in $\finetti{n}$  are known as  \emph{comparative probability orders},  and they enjoy applications in decision theory and economics \cite{kraft,fine,fishburn,slinko}. A comparative probability order $\preceq$ is \emph{additively representable} when there is a probability measure $p : [n] \rightarrow [0,1]$ that induces the order, namely $p(X) \leq p(Y)$ if and only if $X \preceq Y$.

In a more algebraic context, Maclagan \cite{maclagan} refered to total orders in $\finetti{n}$ as   \emph{Boolean term orders}  and studied their combinatorial properties. 
Maclagan introduced a \emph{flip} operation between Boolean term orders, which consists of multiple (related) adjacent transpositions so that (F2) still holds. 
The \emph{flip graph} is the graph with vertex set $\finetti{n}$, where two  orders are adjacent when they differ by one flip.
It is an open question whether the flip graph is connected for $n \geq 9$. Christian et al. \cite{christian} further studied flippable pairs of orders and their relation to the polytope of an additively representable order.

In social choice theory, these total orders are called \emph{completely separable preferences} \cite{hodge,bradley}. In this setting, de Finetti's condition ensures that a voter's preference for the outcomes on a subset $S \subset [n]$ of proposals is independent of the outcome of the proposals in $\overline{S}$. Hodge and TerHaar \cite{hodge3} showed that the number of de Finetti total orders satisfies $n! \cdot |\finetti{n}| = o(\cL(B_n))$. In fact, they proved the stronger condition that linear extensions with at least one pair $X,Y$ of  proper nontrivial subsets satisfying condition (F2) are vanishingly rare. Other research on separable preferences  focuses on the \emph{admissibility problem}: which collection of subsets can occur as the collection of \emph{separable} sets $S$, meaning that (F2) holds for any subsets $X,Y \subset S$ and any $Z \subset \overline{S}$, see \cite{hodge3,hodge4,beveridge+calaway}.

\subsubsection{Gog Triangles and Magog Triangles}

Theorem \ref{thm:fn2} establishes a bijection between the de Finetti refinements $\finetti{n}{2}{1}$  and the magog triangles $\magog{n-1}$.  This connects our poset refinement problem to the illustrious family of
 \emph{alternating sign matrices}. See \cite{bressoud+propp,bressoud}, respectively, for a brief or an extended recounting of the history of  the famous alternating sign matrix conjecture. Magog triangles of $\magog{n}$ are in bijection with totally the symmetric self-complementary plane partitions (TSSCPP) in a $2n \times 2n \times 2n$ box. Andrews \cite{andrews} proved that the number of such TSSCPP is given by equation \eqref{eqn:asm}.
Meanwhile, gog triangles $\gog{n}$ are in bijection with $n \times n$ alternating sign matrices (ASM). Zielberger  \cite{zielberger} proved that $|\magog{n}|=|\gog{n}|$, which confirmed that TSSCPPs and ASMs are equinumerous. Kuperberg \cite{kuperberg} later gave a more streamlined proof using the 6-vertex model from statistical mechanics.  

There are many combinatorial manifestations of the ASM sequence \eqref{eqn:asm},  see \cite{bressoud,propp}. A natural bijective proof between TSSCPPs and ASMs (or equivalently, between magog and gog triangles) remains elusive, though  progress on subfamilies has been achieved \cite{ayyer, biane, striker2018}. Posets and triangular arrays of numbers (such as gog, magog and kagog triangles) continue to play an essential role in ASM and TSSCPP research. Terwilliger \cite{terwilliger} 
defines a poset refinement of the the Boolean lattice $B_n$ whose maximal chains are in natural bijection with ASMs (and hence gog triangles). In comparison, our main result
 puts a family of poset refinements of $B_{n,2}$ in bijection with magog triangles (and hence TSSCPPs). 

In fact, our work may have more in common with Striker \cite{striker2011}, which defines a tetrahedral poset $T_n$ whose subposets trace connections between TSSCPPs, ASMs and other combinatorial sequences.  In particular, $T_n$ has one subposet whose order ideals can be described via families of triangular arrays. The order ideals of one such subposet is in bijection with gog triangles (and hence with ASMs). There are six  distinct subposets whose order ideals (with associated triangular families)  are in bijection with magog triangles (and hence with TSSCPPs). We note that our kagog triangles are not among the triangular families described in \cite{striker2011}, so the family of TSSCPP triangles continues to grow.

%% file: finetti-ballot.tex
\section{Shifted Standard Young Tableau of Staircase Shape and $\finetti{n}{2}$}

\label{sec:finetti-ballot}

This brief section contains a proof of Proposition \ref{thm:fn2}: we give 
a simple bijection between $\finetti{n}{2}$ and shifted standard Young tableaux (shifted SYT) of shape $(n,n-1,\ldots,1)$. 
Figure \ref{fig:fn2} exemplifies the mapping for $n=4$.

\begin{proof}[Proof of Proposition \ref{thm:fn2}.]
To ease exposition, we identify the singleton $\{ i \}$ with the doubleton $\{i,0\}$. 
Ignore the set $\emptyset$ and  lay out the lattice $F_{n,2}$ in a shifted staircase grid so that 
row $k$ contains the sets $\{i,k-1\}$ for $k \leq i \leq n$ in increasing order. This grid induces a shifted staircase Ferrers diagram $(n,n-1, \ldots, 1)$ whose boxes  are indexed the $n(n+1)/2$ nontrivial members of $F_{n,2}$.

\begin{figure}[ht!]

\begin{center}
\begin{tikzpicture}

%
%
%
%
%
%
%
%
%
%
%
%
%
%

\begin{scope}[shift={(3,0)}]

\node (n1) at (0,0) {1};
\node[right of=n1] (n2)  {2};
\node[right of=n2] (n3)  {3};
\node[right of=n3] (n4)  {4};

\node[below of=n2] (n12) {21};
\node[below of=n3] (n13) {31};
\node[below of=n4] (n14) {41};

\node[below of=n13] (n23) {32};
\node[below of=n14] (n24) {42};

\node[below of=n24] (n34) {43};

\draw (n4) -- (n3) -- (n2) -- (n1);

\draw (n2) -- (n12) -- (n13) -- (n14);
\draw  (n13) --  (n23) -- (n24);
\draw (n3) -- (n13);
\draw (n4) -- (n14) -- (n24) -- (n34);

\node at (1.5,-3.75) {(a)};

\end{scope}

%
%

\begin{scope}[shift={(8.5,0)}]

\node (n1) at (0,0) {1};
\node[right of=n1] (n2)  {2};
\node[right of=n2] (n3)  {3};
\node[right of=n3] (n4)  {4};

\node[below of=n2] (n12) {21};
\node[below of=n3] (n13) {31};
\node[below of=n4] (n14) {41};

\node[below of=n13] (n23) {32};
\node[below of=n14] (n24) {42};

\node[below of=n24] (n34) {43};

\draw (n1) -- (n2) -- (n3) -- (n12) -- (n4) -- (n13) -- (n23) -- (n14) -- (n24) -- (n34) ;

\node at (1.5,-3.75) {(b)};

\end{scope}

\begin{scope}[shift={(15,-1.5)}]

\node at (0,0) {$\begin{ytableau}
1 & 2 & 3 & 5 \\
\none & 4 & 6 & 8 \\
\none & \none & 7 & 9 \\
\none & \none & \none & 10
\end{ytableau}
$
};

\node at (0,-2.25) {(c)};

\end{scope}

\end{tikzpicture}
\end{center}

\caption{Mapping $F_{4,2}$ to a shifted standard Young tableau. 
(a) The nontrivial sets in $F_{4,2}$ laid out in a shifted staircase grid.
(b) A de Finetti total order and (c) its corresponding  shifted standard Young tableau. }

\label{fig:fn2}
\end{figure}
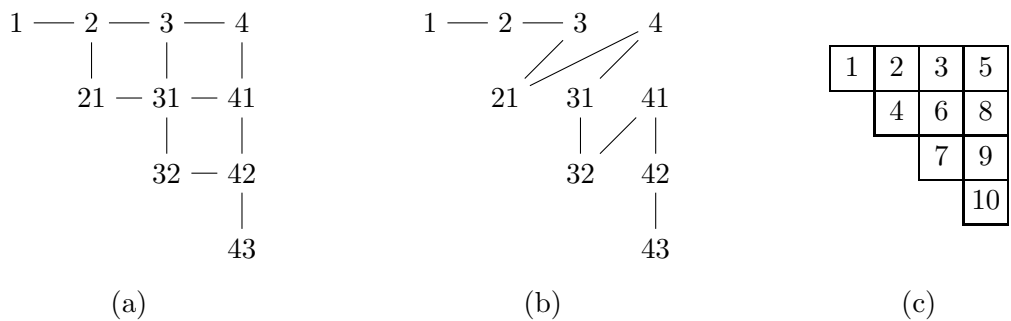

Consider a total order $E \in \finetti{n}{2}$. Place the integer $\ell$ in the box corresponding to the $\ell$th set in total ordering $E$. The result is a  shifted SYT of staircase shape: the rows and columns of the resulting tableau are both increasing because the total ordering satisfies properties (F1) and (F2) of Definition \ref{def:finetti-order}. This mapping is surjective: starting from a shifted SYT, we can reverse the process to find a total order $E \in \finetti{n}{2}$ that maps to it. 
\end{proof}


%% file: kagog-intro.tex
\section{Kagog Triangles, Magog Triangles and $\finetti{n}{2}{1}$}

\label{sec:kagog}

\subsection{Overview}

\label{sec:kagog-overview}

 We  outline of the proof of Theorem \ref{thm:finetti-magog}, deferring the details to the subsections that follow. 
We begin by defining the new family of kagog triangles. This name is a variant of Zielberger's gog and magog terminology, and is meant to invoke their  connection to magog triangles.

\begin{definition}
\label{def:kagog}
A kagog triangle $K$ of index $n$ is an array of nonnegative integers $K(i,j)$ such that 
\begin{itemize}
\item[(K1)] $1 \leq j \leq i \leq n-1$, so the array is triangular;
\item[(K2)]  $0 \leq K(i,j) \leq j$, so entries in column $j$ are at most $j$;
\item[(K3)] $K(i,j) \geq K(i+1,j)$, so columns are weakly decreasing; and
\item[(K4)] if $K(i,j) > 0$ then $K(i,j+1) > K(i,j)$, so rows can start with multiple zeros, but then the positive values are strictly increasing. 
\end{itemize}
We use $\kagog{n}$ to denote  the set of kagog triangles of index $n$.
\end{definition}
The elements of $\kagog{3}$ are
\begin{equation}
\label{eqn:kagog3}
\begin{array}{ccc}
1 &  \\
1 & 2  
\end{array}
\qquad
\begin{array}{ccc}
1 &  \\
0 & 2  
\end{array}
\qquad
\begin{array}{ccc}
1 &  \\
0 & 1  
\end{array}
\qquad
\begin{array}{ccc}
1 &  \\
0 & 0  
\end{array}
\qquad
\begin{array}{ccc}
0 &  \\
0 & 2  
\end{array}
\qquad
\begin{array}{ccc}
0 &  \\
0 & 1  
\end{array}
\qquad
\begin{array}{ccc}
0 &  \\
0 & 0  
\end{array}.
\end{equation}
Note that a kagog triangle of index $n$ only has $n-1$ rows and columns.  
Our first lemma connects kagog triangles $\kagog{n-1}$ to the poset refinements $\finetti{n}{2}{1}$.

\begin{lemma} 
\label{lemma:finetti-kagog}
The set of de Finetti refinements $\finetti{n}{2}{1}$ is in bijection with the set of kagog triangles $\kagog{n-1}$.
\end{lemma}

Next, we turn our attention to the well-known family of magog triangles, which are in bijection with TSSCPPs. 


\begin{definition} 
\label{def:magog}
A magog triangle $M$ of size $n$ is an array of positive integers $M(i,j)$ where
\begin{itemize}
\item[(M1)] $1\leq j\leq i \leq n$, so the array is triangular;
\item[(M2)] $1 \leq M(i,j)\leq j$, so entries in column $j$ are at most $j$; 
\item[(M3)]  $M(i,j)\leq M(i+1,j)$, so columns are weakly increasing; and
\item[(M4)]  $M(i,j)\leq M(i,j+1)$, so rows are weakly increasing.

\end{itemize}
We use $\magog{n}$ to denote  the set of magog triangles of size $n$.
\end{definition}

When a magog triangle is viewed as a Gelfand-Tsetlin triangle of positive integers, conditions (M2) and (M4) are replaced by $M(j,j) \leq j$ and $M(i,j) \leq M(i+1,j+1)$, respectively.
The elements of $\magog{3}$ are
\begin{equation}
\label{eqn:magog3}
\begin{array}{ccc}
1 \\
1 & 1 \\
1 & 1 & 1 \\
\end{array}
\quad
\begin{array}{ccc}
1 \\
1 & 1 \\
1 & 1 & 2
\end{array}
\quad
\begin{array}{ccc}
1 \\ 
1 & 1 \\
1 & 2 & 2
\end{array}
\quad
\begin{array}{ccc}
1 \\ 
1 & 2 \\ 
1 & 2 & 2
\end{array}
\quad
\begin{array}{ccc}
1 \\ 
1 & 1 \\
1 & 1 & 3
\end{array}
\quad
\begin{array}{ccc}
1 \\ 
1 & 1 \\
1 & 2 & 3
\end{array}
\quad
\begin{array}{ccc}
1 \\ 
1 & 2 \\
1 & 2 & 3
\end{array}.
\end{equation}
Our second lemma addresses a duality between kagog triangles $\kagog{n}$ and magog triangles  $\magog{n}$.

\begin{lemma} 
\label{lemma:magog-kagog}
The set of magog triangles $\magog{n}$ is in bijection with the set of kagog triangles $\kagog{n}$.
\end{lemma}

The magog triangles listed in equation \eqref{eqn:magog3} are ordered so that that they biject to the kagog triangles in equation \eqref{eqn:kagog3}.
Also, note that we have chosen to left-justify our triangles (they are often presented using center alignment). This layout choice simplifies our geometric arguments.
The key to proving Lemma \ref{lemma:magog-kagog} is to convert each of these triangles into a pyramid of  stacked cubes, colored gray or white, so that gray cubes cannot appear above white cubes. We offer a generic definition for pyramid construction, which applies to any family $\mathcal{T}_n$ of triangular arrays that form a distributive lattice using the natural partial ordering $T_1 \prec T_2$ whenever $T_1(i,j) \leq T_2(i,j)$ for $1 \leq j \leq i \leq n$. 
This includes  magog triangles $\magog{n}$ and kagog triangles $\kagog{n}$, as well as gog triangles $\gog{n}$ (defined below).

\begin{definition}
\label{def:pyramid}
Let $\mathcal{T}_n$ be a finite distributive lattice of triangular arrays of positive integers $T = T(i,j) $ where $1 \leq j \leq i \leq n$  with minimal triangle $\mini{T}$ and maximal triangle $\maxi{T}$. Define $\pyr{T}$ to be the two-color pyramid of cubes $(i,j,k)$ where $1 \leq i \leq j \leq n$ and $1 \leq k \leq \maxi{T}(i,j)$ where the tower of cubes at $(i,j)$ consists of $T(i,j)$ white cubes below $\maxi{T}(i,j)-T(i,j)$ gray cubes. Define $\pyr{\mathcal{T}_n} = \{ \pyr{T} : T \in \mathcal{T}_n \}$ to be the collection of two-color pyramids.
\end{definition}

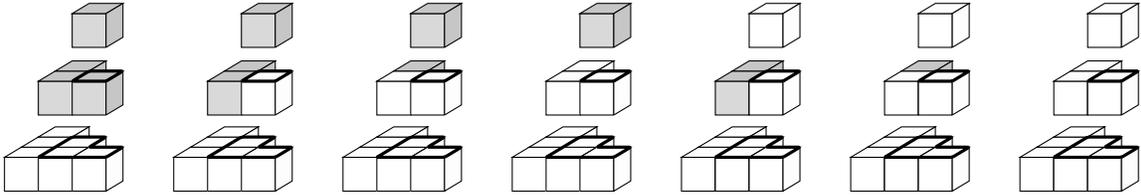
\begin{figure}[ht!]
    \centering
\begin{tikzpicture}[scale=0.45]

\begin{scope}[shift={(0,0)}]
\twocolor{0/3}{0/2,0/1}
\draw[rounded corners = 0.3mm, very thick, black] (0 cm, 1cm) -- (2 cm, 1 cm) -- (2.5 cm,1.3 cm)--(1.5 cm, 1.3 cm) -- (2 cm,1.6 cm)--(1 cm, 1.6cm)--(0 cm, 1 cm);
\begin{scope}[shift={(1,2.25)}]
\twocolor{2/0}{1/0}
\draw[rounded corners = 0.3mm, very thick, black] (1 cm,1)--(0cm,1cm)--(.5cm,1.3cm)--(1.5cm,1.3cm)--(1,1);
\end{scope}
\begin{scope}[shift={(2,4.25)}]
\twocolor{1/0}{}
\end{scope}
\end{scope}

\begin{scope}[shift={(5,0)}]
\twocolor{0/3}{0/2,0/1}
\draw[rounded corners = 0.3mm, very thick, black] (0 cm, 1cm) -- (2 cm, 1 cm) -- (2.5 cm,1.3 cm)--(1.5 cm, 1.3 cm) -- (2 cm,1.6 cm)--(1 cm, 1.6cm)--(0 cm, 1 cm);
\begin{scope}[shift={(1,2.25)}]
\twocolor{1/1}{1/0}
\draw[rounded corners = 0.3mm, very thick, black] (1 cm,1)--(0cm,1cm)--(.5cm,1.3cm)--(1.5cm,1.3cm)--(1,1);
\end{scope}
\begin{scope}[shift={(2,4.25)}]
\twocolor{1/0}{}
\end{scope}
\end{scope}

\begin{scope}[shift={(10,0)}]
\twocolor{0/3}{0/2,0/1}
\draw[rounded corners = 0.3mm, very thick, black] (0 cm, 1cm) -- (2 cm, 1 cm) -- (2.5 cm,1.3 cm)--(1.5 cm, 1.3 cm) -- (2 cm,1.6 cm)--(1 cm, 1.6cm)--(0 cm, 1 cm);
\begin{scope}[shift={(1,2.25)}]
\twocolor{0/2}{1/0}
\draw[rounded corners = 0.3mm, very thick, black] (1 cm,1)--(0cm,1cm)--(.5cm,1.3cm)--(1.5cm,1.3cm)--(1,1);
\end{scope}
\begin{scope}[shift={(2,4.25)}]
\twocolor{1/0}{}
\end{scope}
\end{scope}

\begin{scope}[shift={(15,0)}]
\twocolor{0/3}{0/2,0/1}
\draw[rounded corners = 0.3mm, very thick, black] (0 cm, 1cm) -- (2 cm, 1 cm) -- (2.5 cm,1.3 cm)--(1.5 cm, 1.3 cm) -- (2 cm,1.6 cm)--(1 cm, 1.6cm)--(0 cm, 1 cm);
\begin{scope}[shift={(1,2.25)}]
\twocolor{0/2}{0/1}
\draw[rounded corners = 0.3mm, very thick, black] (1 cm,1)--(0cm,1cm)--(.5cm,1.3cm)--(1.5cm,1.3cm)--(1,1);
\end{scope}
\begin{scope}[shift={(2,4.25)}]
\twocolor{1/0}{}
\end{scope}
\end{scope}

\begin{scope}[shift={(20,0)}]
\twocolor{0/3}{0/2,0/1}
\draw[rounded corners = 0.3mm, very thick, black] (0 cm, 1cm) -- (2 cm, 1 cm) -- (2.5 cm,1.3 cm)--(1.5 cm, 1.3 cm) -- (2 cm,1.6 cm)--(1 cm, 1.6cm)--(0 cm, 1 cm);
\begin{scope}[shift={(1,2.25)}]
\twocolor{1/1}{1/0}
\draw[rounded corners = 0.3mm, very thick, black] (1 cm,1)--(0cm,1cm)--(.5cm,1.3cm)--(1.5cm,1.3cm)--(1,1);
\end{scope}
\begin{scope}[shift={(2,4.25)}]
\twocolor{0/1}{}
\end{scope}
\end{scope}

\begin{scope}[shift={(25,0)}]
\twocolor{0/3}{0/2,0/1}
\draw[rounded corners = 0.3mm, very thick, black] (0 cm, 1cm) -- (2 cm, 1 cm) -- (2.5 cm,1.3 cm)--(1.5 cm, 1.3 cm) -- (2 cm,1.6 cm)--(1 cm, 1.6cm)--(0 cm, 1 cm);
\begin{scope}[shift={(1,2.25)}]
\twocolor{0/2}{1/0}
\draw[rounded corners = 0.3mm, very thick, black] (1 cm,1)--(0cm,1cm)--(.5cm,1.3cm)--(1.5cm,1.3cm)--(1,1);
\end{scope}
\begin{scope}[shift={(2,4.25)}]
\twocolor{0/1}{}
\end{scope}
\end{scope}

\begin{scope}[shift={(30,0)}]
\twocolor{0/3}{0/2,0/1}
\draw[rounded corners = 0.3mm, very thick, black] (0 cm, 1cm) -- (2 cm, 1 cm) -- (2.5 cm,1.3 cm)--(1.5 cm, 1.3 cm) -- (2 cm,1.6 cm)--(1 cm, 1.6cm)--(0 cm, 1 cm);
\begin{scope}[shift={(1,2.25)}]
\twocolor{0/2}{0/1}
\draw[rounded corners = 0.3mm, very thick, black] (1 cm,1)--(0cm,1cm)--(.5cm,1.3cm)--(1.5cm,1.3cm)--(1,1);
\end{scope}
\begin{scope}[shift={(2,4.25)}]
\twocolor{0/1}{}
\end{scope}
\end{scope}

\end{tikzpicture}

\caption{The two-color pyramids from $\pyrmagog{3}$, sliced into horizontal layers. The shadow of each layer is thickly drawn on the layer below.}

\label{fig:magog-pyr}

\end{figure}

Figure \ref{fig:magog-pyr} shows the seven magog pyramids, listed in the same order as  in equation \eqref{eqn:magog3}. To facilitate visualization, the pyramids have been sliced into layers of equal height.
This two-color pyramid mapping is a variation of the standard interpretation triangular array $T$ as a stack of cubes where the tower at $(i,j)$ has height $T(i,j)$. Indeed, we can view the white cubes as present and the gray cubes as absent. In our proof, tracking the absent cubes is essential, so the two-color pyramids are more illuminating.
Intuitively, the bijection from magog triangles to kagog  triangles corresponds to removing the bottom layer of the magog pyramid, then swapping the colors of the cubes and finally performing an appropriate affine transformation.

\begin{proof}[Proof of Theorem \ref{thm:finetti-magog}]
Follows immediately from Lemma \ref{lemma:finetti-kagog} and Lemma \ref{lemma:magog-kagog}.
\end{proof}

We prove Lemma \ref{lemma:finetti-kagog} and Lemma \ref{lemma:magog-kagog} in the next two subsections.

%% file: finetti-kagog.tex
\subsection{The bijection from $\finetti{n}{2}{1}$ to  $\cK_{n-1}$}

In this subsection, we prove Lemma \ref{lemma:finetti-kagog}.
 Figure \ref{fig:345} shows the de Finetti lattice $F_{n,2}$ for  $n=3,4,5$ and also indicates the sublattice
\begin{equation}
\label{eqn:undetermined}
I_k = \left\{ \{j,i\} \mid  1 \leq i < j< k \right\} 
\end{equation}
of doubletons that are incomparable with   singleton $ \{k \}$.

\begin{figure}[ht!]

\begin{center}
\begin{tikzpicture}

\begin{scope}[shift={(0,0)}]




\draw[gray, dashed] (.7, 2.1) circle (.3);

\node at (1.45,2.35) {$I_3$};

\node (n0) at (0,0) {$\emptyset$};
\node[above right of=n0] (n1)  {1};
\node[above left of=n1] (n2)  {2};
\node [above left of=n2]  (n3)  {3};

\node[ above right  of=n2] (n12) {21};
\node[ above right of=n3] (n13) {31};
\node[ above right of=n13] (n23) {32};

\draw (n3) -- (n2) -- (n1) -- (n0);

\draw (n2) -- (n12);
\draw (n3) -- (n13);
\draw (n12) -- (n13) -- (n23);

\node at (-1.25,.75) {$n=3$};

\end{scope}

\begin{scope}[shift={(5,0)}]

\draw[gray, dashed,rounded corners=10pt] (-.45,2.8) -- (1,4.25) -- (1,1.45) --cycle;
\node at (1.45,2.35) {$I_4$};

\node (n0) at (0,0) {$\emptyset$};;
\node[above right of=n0]  (n1)  {1};
\node[above left of=n1] (n2)  {2};
\node[above left of=n2] (n3)  {3};
\node[above left of=n3] (n4)  {4};

\node[above right of=n2] (n12) {21};
\node[above right of=n3] (n13) {31};
\node[above right of=n13] (n23) {32};

\node[above right of=n4] (n14) {41};
\node[above right of=n14] (n24) {42};
\node[above  right of=n24] (n34) {43};

\draw (n4) -- (n3) -- (n2) -- (n1) -- (n0);

\draw (n2) -- (n12) -- (n13) -- (n14) -- (n24);
\draw  (n13) --  (n23) -- (n24) -- (n34);
\draw (n3) -- (n13);
\draw (n4) -- (n14);

\node at (-1.25,.75) {$n=4$};
\end{scope}

\begin{scope}[shift={(10,0)}]

\draw[gray, dashed,rounded corners=10pt] (-1.15,3.5) -- (1,5.65) -- (1,1.35) --cycle;
\node at (1.45,2.35) {$I_5$};

\node (n0) at (0,0) {$\emptyset$};;
\node[above right of=n0]  (n1)  {1};
\node[above left of=n1] (n2)  {2};
\node[above left of=n2] (n3)  {3};
\node[above left of=n3] (n4)  {4};
\node[above left of=n4] (n5)  {5};

\node[above right of=n2] (n12) {21};
\node[above right of=n3] (n13) {31};
\node[above right of=n4] (n14) {41};
\node[above right of=n5] (n15) {51};

\node[above right of=n13] (n23) {32};
\node[above right of=n14] (n24) {42};
\node[above right of=n15] (n25) {52};

\node[above  right of=n24] (n34) {43};
\node[above  right of=n25] (n35) {53};
\node[above  right of=n35] (n45) {54};

\draw (n5) -- (n4) -- (n3) -- (n2) -- (n1) -- (n0);

\draw (n2) -- (n12) -- (n13) -- (n14) -- (n15);
\draw  (n13) --  (n23) -- (n24) -- (n25);
\draw (n3) -- (n13);
\draw (n4) -- (n14) -- (n24) -- (n34) -- (n35);
\draw (n5) -- (n15) -- (n25) -- (n35) -- (n45);

\node at (-1.25,.75) {$n=5$};

\end{scope}

\end{tikzpicture}
\end{center}

\caption{The lattice $F_{n,2}$  induced by $1 \prec 2 \prec \cdots \prec n$ and de Finetti's condition for $n=3,4,5$. The set $I_n$  contains the doubletons whose comparison with the singleton $n$ is not determined by de Finetti's condition. }

\label{fig:345}
\end{figure}
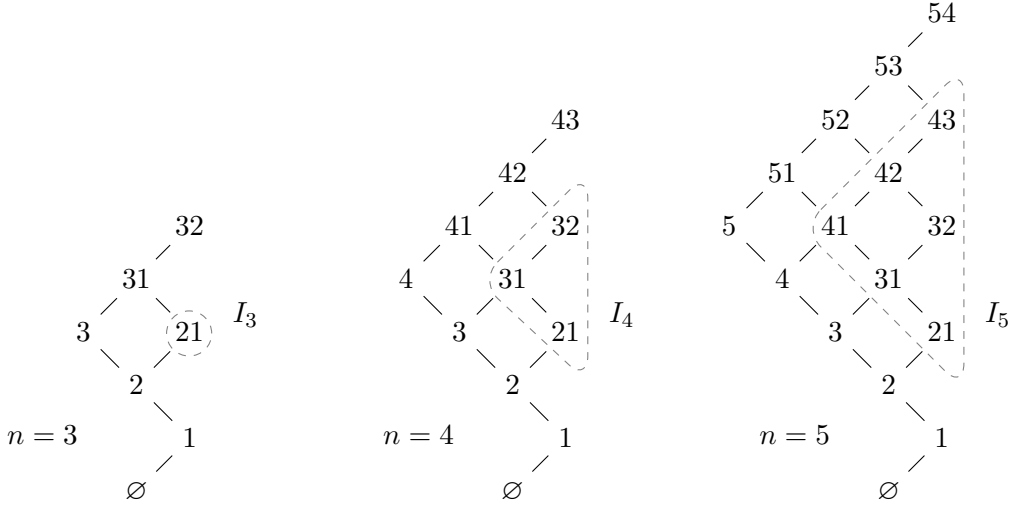

For $k \geq 3$, let $\Phi(I_k)$ be the collection of de Finetti extensions of $I_k \cup \{ k \}$ for which the singleton $\{ k \}$ is comparable with every doubleton of $I_k$ (and no additional extraneous relations). When we restrict a poset extension $E \in \finetti{n}{2}{1}$ to the set $I_k \cup \{ k \}$, we obtain some  $E_k \in \Phi(I_k)$. Similarly, we can induce a unique poset extension $E$ of $F_{n,2}$ from  a list $(E_3, E_4, \ldots, E_n)$ where $E_k \in \Phi(I_k)$. We will have $E \in \finetti{n}{2}{1}$  provided that the union of these orderings does not violate de Finetti's condition (F2). Figure \ref{fig:Fn21example} gives an example of a poset extention $E \in \finetti{n}{2}{1}$ and its collection of $E_k \in \Phi(I_k)$.

\begin{figure}[ht!]
    \centering
\begin{tikzpicture}

\node (n0) at (0,0) {$\emptyset$};
\node[above right of=n0]  (n1)  {1};
\node[above left of=n1] (n2)  {2};
\node[above left of=n2] (n12) {21};
\node[above right of=n12] (n3)  {3};
\node[above right of=n3] (n13) {31};
\node[above left of=n13] (n4)  {4};
\node[above left of=n4] (n14) {41};
\node[above right of=n14] (n5)  {5};

\node[above left of=n5] (n15) {51};

\node[above right of=n4] (n23) {32};
\node[above right of=n5] (n24) {42};
\node[above right of=n15] (n25) {52};

\node[above  right of=n24] (n34) {43};
\node[above  right of=n25] (n35) {53};
\node[above  right of=n35] (n45) {54};

\foreach \i/\j in {0/1,1/2,2/12,3/12,3/13,4/13,4/14,4/23,14/5,23/5,5/15,5/24,15/25,24/25,24/34,25/35,34/35,35/45}
\draw (n\i)--(n\j);

\node[left of=n3, node distance=1.4cm] (l3) {};
\node[right of=n3, node distance=1.4cm] (r3)  {};

\node[left of=n4, node distance=1.4cm] (l4) {};
\node[right of=n4, node distance=1.4cm] (r4)  {};

\node[left of=n5, node distance=1.4cm] (l5) {};
\node[right of=n5, node distance=1.4cm] (r5)  {};

\node [below right of=r3, node distance=0.5cm] (e) {$E$};

\begin{scope}[shift={(6,5)}]
\node (n12) at (0,0) {21};
\node[above right of=n12] (n13) {31};
\node[above left of=n13] (n14) {41};
\node[above right of=n13] (n23) {32};
\node[above right of=n14] (n5)  {5};
\node[above right of=n5] (n24) {42};
\node[above  right of=n24] (n34) {43};

\node [right of=n23, node distance=1.2cm] (e) {$E_5$};

\foreach \i/\j in {12/13,13/14,13/23,14/5,23/5,5/24,24/34}
\draw (n\i)--(n\j);
\end{scope}

\begin{scope}[shift={(6,2.2)}]
\node (n12) at (0,0) {21};
\node[above right of=n12] (n13) {31};
\node[above right of=n13] (n4) {4};
\node[above right of=n4] (n23) {32};

\node [right of=n13, node distance=2cm] (e) {$E_4$};

\foreach \i/\j in {12/13,13/4,4/23}
\draw (n\i)--(n\j);
\end{scope}

\begin{scope}[shift={(6,0)}]
\node (n12) at (0,0) {21};
\node[above right of=n12] (n3) {3};

\node [right of=n3, node distance=2cm] (e) {$E_3$};

\foreach \i/\j in {12/3}
\draw (n\i)--(n\j);
\end{scope}

\node (a) at (1,-1) {(a)};
\node (b) at (7,-1) {(b)};

\end{tikzpicture}
    \caption{(a) A poset extension $E$ from $\finetti{n}{2}{1}$. Each singleton is comparable with every other set. (b) The subposets $E_3$, $E_4$ and $E_5$ of $E$.}
    \label{fig:Fn21example}
\end{figure}
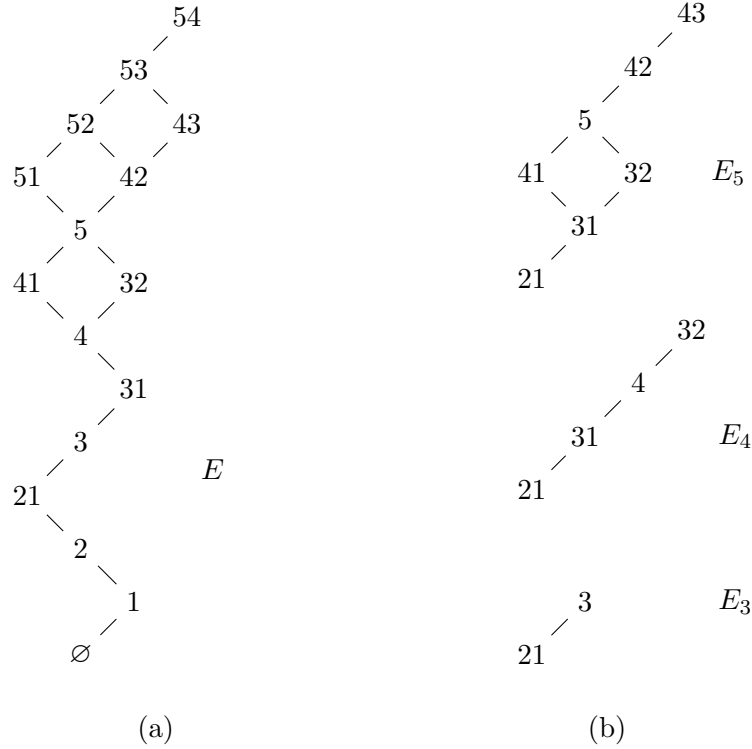

Our  bijection from the poset extensions of $F_{n,2}$ in $\finetti{n}{2}{1}$ to the kagog triangles in $\kagog{n-1}$ proceeds as follows. Given a de Finetti extension $E \in \finetti{n}{2}{1}$, we create the corresponding list $(E_3, \ldots E_n)$ where $E_k \in \Phi(I_k)$. We then map extension $E$ to a kagog triangle $K \in \kagog{n-1}$ so that extension $E_j$ maps to row $j-2$ of triangle $K$ for $3 \leq j \leq n$. The row constraint (K3) of the kagog triangle will correspond to the internal structure of each $E_k$. The column constraint (K2) of the kagog triangle will correspond to having singleton de Finetti extensions $(E_3, E_4, \ldots, E_n)$ whose union also abides by de Finetti's condition.

We begin by introducing a convenient $k$-list version of the power set $\cP([k])$. Let
$$
\cL([k]) = \big\{ 
(\underbrace{0, \ldots, 0}_{k-j}, s_1, s_2, \ldots , s_j ) 
\mid 0 \leq j \leq k \mbox{ and } 1 \leq s_1 < s_2 < \cdots < s_j \leq k
\big\}
$$
be the set of $k$-lists produced by listing the elements of $S \subset [k]$ in increasing order and then prepending $k - |S|$ zeros.

\begin{lemma}
Each row $1 \leq k \leq n-1$ of a kagog triangle in $\kagog{n}$ is an element of $\cL([k])$.
\end{lemma}

\begin{proof}
The constraint (K4) on row $k$ of  a kagog triangle in $\kagog{n}$  is identical to the conditions on a list in $\cL([k])$. 
\end{proof}

\begin{lemma}
\label{lem:template}
For $n \geq 3$,   $\Phi(I_n)$  is in bijection with $\cL([n-2])$.
\end{lemma}

Let us build some intuition with two examples. First, we consider extensions in $\Phi(I_5)$. 
We must determine the comparisons of the singleton $\{5\}$ with the doubletons in the lattice $I_5$. By interweaving empty boxes among the doubletons, we obtain the template 
\begin{center}
\begin{tikzpicture}[scale=.3]

\begin{scope}

\node at (0,0) 
{$\displaystyle
\begin{array}{ccccccc}
\Box& 41& \Box &  42 & \Box & 43 & \Box \\
\Box& 31 & \Box &  32 & \Box \\
 \Box &  21 & \Box
\end{array}
$};

\node at (12,0) {$\longleftrightarrow$};

\end{scope}

\begin{scope}[shift={(17,-4.5)}]

\sq{0}{5}
\sq{1}{5}
\sq{2}{5}
\sq{3}{5}
\sq{0}{4}
\sq{1}{4}
\sq{2}{4}
\sq{0}{3}
\sq{1}{3}

\end{scope}
\end{tikzpicture}
\end{center}
where omitting the doubletons gives the Ferrers diagram for the integer partition  $ (4,3,2)$. Specifying the comparisons with singleton $\{5\}$ is equivalent to placing a dot in each row of $ (4,3,2)$.
Looking only at the top row, placing a 5 in the first box
\begin{center}
\begin{tikzpicture}[scale=.3]

\begin{scope}

\sq{0}{5}
\sq{1}{5}
\sq{2}{5}
\sq{3}{5}
\sq{0}{4}
\sq{1}{4}
\sq{2}{4}
\sq{0}{3}
\sq{1}{3}

\draw[very thick] (0,5) -- (3,5) -- (3,4) -- (2,4) -- (2,3) -- (0,3) -- cycle;

\node at (0.5, 5.5) {$\bullet$};

\end{scope}
\end{tikzpicture}

\end{center}
means that $5 \prec 41$. This  puts no further de Finetti  restrictions on the remaining two rows. 
The four ways to complete this configuration are

\begin{center}
\begin{tikzpicture}[scale=.3]

\begin{scope}

\sq{0}{5}
\sq{1}{5}
\sq{2}{5}
\sq{3}{5}
\sq{0}{4}
\sq{1}{4}
\sq{2}{4}
\sq{0}{3}
\sq{1}{3}

\node at (0.5, 5.5) {$\bullet$};
\node at (0.5, 4.5) {$\bullet$};
\node at (0.5, 3.5) {$\bullet$};

\end{scope}

\begin{scope}[shift={(7,0)}]

\sq{0}{5}
\sq{1}{5}
\sq{2}{5}
\sq{3}{5}
\sq{0}{4}
\sq{1}{4}
\sq{2}{4}
\sq{0}{3}
\sq{1}{3}

\node at (0.5, 5.5) {$\bullet$};
\node at (0.5, 4.5) {$\bullet$};
\node at (1.5, 3.5) {$\bullet$};

\end{scope}

\begin{scope}[shift={(14,0)}]

\sq{0}{5}
\sq{1}{5}
\sq{2}{5}
\sq{3}{5}
\sq{0}{4}
\sq{1}{4}
\sq{2}{4}
\sq{0}{3}
\sq{1}{3}

\node at (0.5, 5.5) {$\bullet$};
\node at (1.5, 4.5) {$\bullet$};
\node at (1.5, 3.5) {$\bullet$};

\end{scope}

\begin{scope}[shift={(21,0)}]

\sq{0}{5}
\sq{1}{5}
\sq{2}{5}
\sq{3}{5}
\sq{0}{4}
\sq{1}{4}
\sq{2}{4}
\sq{0}{3}
\sq{1}{3}

\node at (0.5, 5.5) {$\bullet$};
\node at (2.5, 4.5) {$\bullet$};
\node at (1.5, 3.5) {$\bullet$};

\end{scope}

\end{tikzpicture}
\end{center}
which correspond to the comparisons 
$$
5 \prec 32 \qquad 21 \prec 5 \prec 31 \qquad 31 \prec 5 \prec 32 \qquad 32 \prec 5 \prec 41.
$$

Our final step is to count the boxes to the right of these dots, starting from the bottom row and moving up. This results in the lists $(1,2,3)$, $(0,2,3)$,  $(0,1,3)$ and $ (0,0,3)$
from $\cL([3])$.

Next, we consider extensions in $\Phi(I_7)$. Specifying the comparisons of singleton $\{ 7 \}$ with the doubletons in $I_7$ is equivalent to placing a  dot in each row of the integer partition $(6,5,4,3,2)$. Suppose that we place a $7$ in the third box of the first row, corresponding to  $62 \prec 7 \prec 63$.  Now de Finetti's condition
leads to $21 \prec 32 \prec 42 \prec 52  \prec 62 \prec 7$, which yields  the partially filled diagram
\begin{center}
\begin{tikzpicture}[scale=.3]

\begin{scope}

\sq{0}{7}
\sq{1}{7}
\sq{2}{7}
\sq{3}{7}
\sq{4}{7}
\sq{5}{7}
\sq{0}{6}
\sq{1}{6}
\sq{2}{6}
\sq{3}{6}
\sq{4}{6}
\sq{0}{5}
\sq{1}{5}
\sq{2}{5}
\sq{3}{5}
\sq{0}{4}
\sq{1}{4}
\sq{2}{4}
\sq{0}{3}
\sq{1}{3}

\node at (2.5, 7.5) {$\bullet$};
\node at (2.5, 4.5) {$\bullet$};
\node at (1.5, 3.5) {$\bullet$};

\draw[very thick] (2,7) -- (5,7) -- (5,6) -- (4,6) -- (4,5) -- (2,5) -- cycle;


\end{scope}
\end{tikzpicture}

\end{center}
which contains a shifted copy of partition $(3,2)$ whose rows must each be assigned a dot. 
This can be done in four ways, and counting the boxes to the right of the dots gives the lists $(0,0,1,2,3)$, $(0,0,0,2,3)$,  $(0,0,0,1,3)$ and $(0,0,0,0,3)$
from $\cL([5])$.
We now prove Lemma \ref{lem:template} by strong induction.

\begin{figure}[ht!]
\begin{center}

\begin{tikzpicture}[scale=.33]

\begin{scope}[shift={(-15,0)}]

\foreach \i in {0,1,...,8} {
	\draw[thick] (\i, \i) -- (\i+1, \i) -- (\i+1, \i+1);
}

\draw[thick] (0,0) -- (-1,0) -- (-1,9) -- (9,9);

\node at (2,5) {\small $\lambda_{n-1}$};

\node at (4.5,-1) {\small (a)};

\end{scope}

%
%
%
%
%
%
%
%
%
%

\begin{scope}[shift={(0,0)}]

\fill[gray!25] (2,8) -- (3,8) -- (3,9) -- (2,9);
\fill[gray!25] (3,8) -- (9,8) -- (9,9) -- (3,9);

\fill[gray!25] (-1,0) -- (-1,9) -- (2,9) -- (2,2) -- (2,1) -- (1,1) -- (1,0) -- cycle;

\fill[gray!25] (2,2) -- (2,3) -- (3,3) -- (3,2) -- cycle;

\foreach \i in {0,1,...,8} {
	\draw[thick] (\i, \i) -- (\i+1, \i) -- (\i+1, \i+1);
}

\draw[thick] (0,0) -- (-1,0) -- (-1,9) -- (9,9);
\draw (-1,8) -- (8,8);


\draw (2,3) -- (3,3);

\draw (3,8) -- (3,9);
\draw (2,3) -- (2,9);

\draw[fill] (2.5,8.5) circle (6pt);
\draw[fill] (2.5,2.5) circle (6pt);
\draw[fill] (1.5,1.5) circle (6pt);
\draw[fill] (.5,.5) circle (6pt);

\node at (6,8,5) {\small $\lambda_{n-k-1}$};

\node at (4.5,-1) {\small (c)};

\end{scope}

\begin{scope}[shift={(15,0)}]

\fill[gray!25] (2,8) -- (3,8) -- (3,9) -- (2,9);

\fill[gray!25] (-1,0) -- (-1,9) -- (3,9) --(3,8) -- (4,8) -- (4,7) -- (5,7) -- (5,4) --
    (4,4) -- (4,3) -- (3,3) -- (2,3) -- (2,2) -- (2,1) -- (1,1) -- (1,0) -- cycle;

\fill[gray!25] (2,2) -- (2,3) -- (3,3) -- (3,2) -- cycle;

\foreach \i in {0,1,...,8} {
	\draw[thick] (\i, \i) -- (\i+1, \i) -- (\i+1, \i+1);
}

\draw[thick] (0,0) -- (-1,0) -- (-1,9) -- (9,9);
\draw (3,8) -- (8,8);
\draw (5,7) -- (7,7);
\draw (5,6) -- (6,6);


\draw (3,8) -- (3,9);
\draw (-1,9) -- (3,9) --(3,8) -- (4,8) -- (4,7) -- (5,7) -- (5,4);

\draw (4,9) -- (4,8);
\draw (5,9) -- (5,7);
\draw (6,9) -- (6,6);
\draw (7,9) -- (7,7);
\draw (8,9) -- (8,8);

\draw[fill] (2.5,8.5) circle (6pt);
\draw[fill] (3.5,7.5) circle (6pt);
\draw[fill] (4.5,6.5) circle (6pt);
\draw[fill] (4.5,5.5) circle (6pt);
\draw[fill] (4.5,4.5) circle (6pt);
\draw[fill] (3.5,3.5) circle (6pt);
\draw[fill] (2.5,2.5) circle (6pt);
\draw[fill] (1.5,1.5) circle (6pt);
\draw[fill] (.5,.5) circle (6pt);

\node at (4.5,-1) {\small (c)};

\end{scope}

\end{tikzpicture}

\end{center}
\caption{(a) The Ferrers diagram $\lambda_{n-1}=(n-1,n-2, \ldots, 2)$ for ordering singleton $\{n\}$ with the doubletons.
(b) When $\{ n-1, k-1\} \prec n \prec \{n-1, k \}$, we place a dot in the $k$th position. This places de Finetti restrictions on the remaining rows. Completing the order is equivalent to choosing an order inside tempate $\lambda_{n-k-1}$.
(c) Counting the boxes to the right of the dots gives the subset $ \{ 6, 4, 2,1 \} \subset [n-2]$.
}

\label{fig:template}

\end{figure}
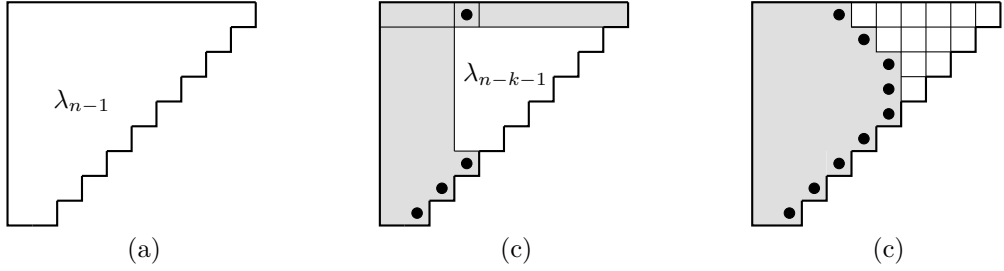

\begin{proof}[Proof of Lemma \ref{lem:template}.]
We recursively define the bijection $f : \Phi(I_n) \rightarrow \cL([n-2])$. 
For $n=3$, we map the ordering with $21 \prec 3$ to the list $( 0 )$ and the ordering with $3 \prec 21$ to the list $(1)$. 
Assume that we have specified the bijection  $f : \Phi(I_\ell) \rightarrow \cL([\ell-2])$ for $2 \leq \ell < n$. 
We determine the image of an extension  $E \in \Phi(I_n)$. 
As in the above examples, we represent $E$ by placing a dot in each row of the Ferrers diagram $\lambda_{n-1} = (n-1, n-2, \ldots, 2),$ see Figure \ref{fig:template}(a).  Our target list in $L \in \cL( [n-2])$ will be obtained by counting the boxes to the right of the dot in each row.

Placing a dot in position $1 \leq k \leq n-1$ of the first row of the template 
$$ 
\Box \quad \{ n-1, 1\} \quad \Box \quad  \{ n-1, 2\} \quad \Box \quad \cdots \quad \Box \quad \{ n-1, n-3\} \quad \Box \quad  \{ n-1, n-2\} \quad \Box
$$
resolves the ordering of singleton $\{n \}$ with the doubletons  $\{ n-1, j\}$ for $1 \leq j \leq n-2$.
If $k=1$ then $\{n\} \prec \{ n-1, 1 \}$; if $1 < k < n-1$ then $ \{ n-1, k-1 \} \prec \{ n \} \prec \{n-1, k \}$; and if $k=n-1$ then $\{ n-1, n-2 \} \prec \{ n \}.$
De Finetti's condition (F2) puts constraints on the remaining rows. For rows $2 \leq i \leq n-k-1$, we must place $n$ in position $k$ or higher. For rows $i > n-k-1$, we must place $n$ in the rightmost (diagonal) position. Therefore, we can restrict our attention to rows 
$2 \leq i \leq n-k-1$ and positions $k \leq j \leq n-2$. But this is simply a translation of the mapping $f: \Phi(I_{n-k}) \rightarrow \cL([n-k-2])$ via a copy of $\lambda_{n-k-1}$, see Figure \ref{fig:template}(b). 
Let $(a_1, a_2, \ldots, a_{n-k-2})  \in \cL( [n-k-2])$ be the image of this mapping. We set 
$$
f(E) = 
(\underbrace{0, \ldots, 0}_{k}, a_1, a_2, \ldots, a_{n-k-2}, n-k-1).
$$
The values in this list are the number of boxes to the right of the dots in Figure \ref{fig:template}(c), when ordered from bottom to top.
\end{proof}

We can now prove that the set of de Finetti extensions $\finetti{n}{2}{1}$ is in bijection with the set of kagog triangles $\cK_{n-1}$.

\begin{proof}[Proof of Lemma \ref{lemma:finetti-kagog}.]
Let $E \in \finetti{n}{2}{1}$ be a de Finetti extension of $F_{n,2}$ so that every singleton is universally comparable in $E$. Consider $(E_3, E_4, \ldots, E_n)$ where $E_k \in \Phi(I_k)$ is the poset extension of $I_k \cup \{ k \}$ induced by $E$.  
Create a triangular array $T = T(i,j) $ for $1 \leq j \leq i \leq n-2$ by applying the mapping $f$ from Lemma \ref{lem:template} to each element in this list of extensions, using the indexing convention 
$$f(E_k) = \big( T(k-2,1),  T(k-2,2), \cdots , T(k-2,k-2) \big), 
\quad 3 \leq k \leq n.$$  
By Lemma \ref{lem:template}, each row satisfies the kagog row constraint.
Meanwhile, the extension $E$ satisfies de Finetti's condition (F2). In particular, for
any $1 \leq i < j < k \leq n$,  if $ \{ k \}  \prec \{ j,i \} $ then  $\{k-1\} \prec \{j,i\}$. 
In terms of triangle $T$, this means that $T(k-2,n-j) \geq T(k-3,n-j)$. 
For $2 \leq j \leq n-1$, this is precisely the contraint that column $n-k$ of a kagog triangle must be weakly decreasing constraint on column. Column $n-1$ has a single entry, so the final column is (vacuously) weakly increasing.
\end{proof}

%% file: kagog-magog.tex

\subsection{The bijection from $\magog{n}$ to $\kagog{n}$}

\label{sec:magog-kagog}

We now prove Lemma \ref{lemma:magog-kagog}. 
Recall that each triangular family $\mathcal{T}_n$ forms a distributive lattice and that Definition \ref{def:pyramid} constructs two-color pyramids in relation to the maximum and minimum triangle of $\mathcal{T}_n$. 

The minimum magog triangle has  $\mini{M}(i,j)=1$ for every entry $(i,j)$ and the maximum magog triangle has $\maxi{M}(i,j)=j$ for every entry $(i,j)$. Our first transformation is to subtract $\mini{M}$ from each magog triangle. The rightmost column becomes all-zero, so we omit it and reindex. This leads to the family of omagog triangles (short for ``zeroed-magog'' triangles).

\begin{definition} 
\label{def:omagog}
An omagog triangle $\oo{M}$ of index $n$ is an array of nonnegative integers $\oo{M}(i,j)$ such that
\begin{itemize}
    \item[(OM1)] $1\leq j\leq i \leq n-1$,  so the array is triangular
    \item[(OM2)] $\oo{M}(i,j)\leq j$, so the entries in column $k$ are at most j;
    \item[(OM3)] $\oo{M}(i,j)\leq \oo{M}(i+1,j)$, so columns are weakly increasing; and
    \item[(OM4)] $\oo{M}(i,j)\leq \oo{M}(i,j+1)$, so rows are weakly increasing.
\end{itemize}
We use $\omagog{n}$ to denote the set of all omagog triangles  of index $n$.
\end{definition}
The set $\omagog{3}$ appears in Figure \ref{o3}, with elements ordered so that they biject to the magog triangles of equation \eqref{eqn:magog3}.
The minimum omagog triangle  satisfies $\mini{\oo{M}}(i,j)=0$ and the maximum omagog triangle  satisfies  $\maxi{\oo{M}}(i,j)=j$ for all entries $(i,j)$. 


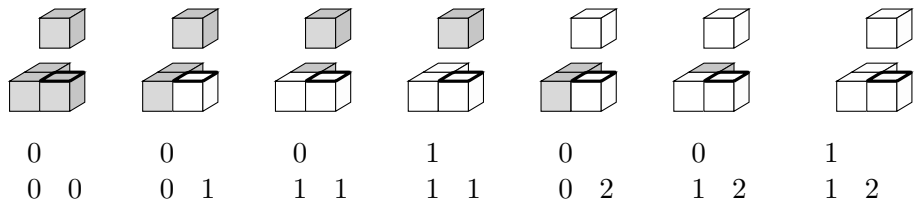
\begin{figure}[ht!]
\centering
\begin{tabular}{ccccccc}

\begin{tikzpicture}[scale=0.4]
\twocolor{2/0}{1/0}
\draw[rounded corners = 0.3mm, very thick, black] (1 cm,1)--(0cm,1cm)--(.5cm,1.3cm)--(1.5cm,1.3cm)--(1,1);
\begin{scope}[shift={(1,2.1)}]
\twocolor{1/0}{}
\end{scope}
\node (l) at (0.5,-2) {$\begin{array}{cc}
 0 \\[-1 mm]0 & 0
\end{array}$};
\end{tikzpicture}

&

\begin{tikzpicture}[scale=0.4]
\twocolor{1/1}{1/0}
\draw[rounded corners = 0.3mm, very thick, black] (1 cm,1)--(0cm,1cm)--(.5cm,1.3cm)--(1.5cm,1.3cm)--(1,1);
\begin{scope}[shift={(1,2.1)}]
\twocolor{1/0}{}
\end{scope}
\node (l) at (0.5,-2) {$\begin{array}{cc}
 0 \\[-1 mm] 0 & 1
\end{array}$};
\end{tikzpicture}

&

\begin{tikzpicture}[scale=0.4]
\twocolor{0/2}{1/0}
\draw[rounded corners = 0.3mm, very thick, black] (1 cm,1)--(0cm,1cm)--(.5cm,1.3cm)--(1.5cm,1.3cm)--(1,1);
\begin{scope}[shift={(1,2.1)}]
\twocolor{1/0}{}
\end{scope}
\node (l) at (0.5,-2) {$\begin{array}{cc}
 0 \\[-1 mm]1 & 1
\end{array}$};
\end{tikzpicture}

&

\begin{tikzpicture}[scale=0.4]
\twocolor{0/2}{0/1}
\draw[rounded corners = 0.3mm, very thick, black] (1 cm,1)--(0cm,1cm)--(.5cm,1.3cm)--(1.5cm,1.3cm)--(1,1);
\begin{scope}[shift={(1,2.1)}]
\twocolor{1/0}{}
\end{scope}
\node (l) at (0.5,-2) {$\begin{array}{cc}
 1 \\[-1 mm]1 & 1
\end{array}$};
\end{tikzpicture}

&

\begin{tikzpicture}[scale=0.4]
\twocolor{1/1}{1/0}
\draw[rounded corners = 0.3mm, very thick, black] (1 cm,1)--(0cm,1cm)--(.5cm,1.3cm)--(1.5cm,1.3cm)--(1,1);
\begin{scope}[shift={(1,2.1)}]
\twocolor{0/1}{}
\end{scope}
\node (l) at (0.5,-2) {$\begin{array}{cc}
0 \\[-1 mm]0 & 2
\end{array}$};
\end{tikzpicture}

&

\begin{tikzpicture}[scale=0.4]
\twocolor{0/2}{1/0}
\draw[rounded corners = 0.3mm, very thick, black] (1 cm,1)--(0cm,1cm)--(.5cm,1.3cm)--(1.5cm,1.3cm)--(1,1);
\begin{scope}[shift={(1,2.1)}]
\twocolor{0/1}{}
\end{scope}
\node (l) at (0.5,-2) {$\begin{array}{cc}
 0 \\[-1 mm]1& 2
\end{array}$};
\end{tikzpicture}

&

\begin{tikzpicture}[scale=0.4]
\blocklayer{2}{1/0}

\begin{scope}[shift={(1,0)}]
\draw[rounded corners = 0.3mm, very thick, black] (1 cm,1)--(0cm,1cm)--(.5cm,1.3cm)--(1.5cm,1.3cm)--(1,1);
\end{scope}
\begin{scope}[shift={(1,2.1)}]
\blocklayer{1}{}
\end{scope}
\node (l) at (0.5,-2) {$\begin{array}{cc}
1 \\[-1 mm]1 & 2
\end{array}$};
\end{tikzpicture}



\end{tabular}
\caption{The omagog triangles $\omagog{3}$ and the corresponding two-color omagog pyramids $\pyromagog{3}$. The pyramids are sliced into horizontal layers. The shadow of a layer is thickly drawn on the layer below.}
\label{o3}
\end{figure}

\begin{proof}[Proof of Lemma \ref{lemma:magog-kagog}.]
We create a bijection $\psi$ from omagog pyramids $\pyromagog{n}$ to kagog pyramids $\pyrkagog{n}$ via a sequence of elementary transformations.

Recall that a two-color pyramid $\pyr{T}$ is a collection of cubes $(i,j,k)$ that are colored white or gray. Renaming these colors as color 1 and color 0, respectively, then the two-color pyramid becomes a binary function on the set of admissible coordinates, that is  $\pyro{M} : (i,j,k) \mapsto \{0,1\}.$
Viewing $\pyro{M}$ as a function allows us to describe the collection $\pyromagog{n}$ of two-color pyramids with a system of inequalities.
We have 
\begin{itemize}
\item $\pyro{M}(i,j,k)$ is defined for $1\leq k\leq j \leq i \leq n-1$.
\item $\pyro{M}(i,j,k)\leq \pyro{M}(i+1,j,k)$: the columns of the magog triangle are nondecreasing,
\item $\pyro{M}(i,j,k) \leq \pyro{M}(i,j+1,k)$: the rows of the magog triangle are nondecreasing, and
\item $\pyro{M}(i,j,k+1) \leq \pyro{M}(i,j,k)$: color 1 (white, present) cubes are below color 0 (gray, absent) cubes, so the cubes that are present obey ``gravity.''
\end{itemize}

We now perform our four step transformation $\psi$. 
\begin{itemize} 
\item \textbf{Step 1:} Invert the colors, or exchange color 0 for color 1 and vice versa. This reverses the inequalities.
\item \textbf{Step 2:} Push all cubes north in their respective column so that row 1 has length $n-1$. This is equivalent to moving the cube $(i,j,k)$ to $(i-(j-1),j,k)$. 
\item \textbf{Step 3:} Tip the entire stack over the $y$-axis via a clockwise rotation by $\pi/2$. This is equivalent to moving the cube $(i,j,k)$ to $(n-k,j,i)$.
\item \textbf{Step 4:} Reflect the stack through the plane $y=(n+1)/2$. This is equivalent to moving the cube $(i,j,k)$ to $(i,n-j,k)$. 
\end{itemize}
After composing these four steps, cube $(i,j,k)$ switches color and moves to $(n-k,n-j,i-j+1)$.
 Figure \ref{fig:M2K} shows the mapping $\psi$ for an omagog pyramid in $\pyromagog{4}$.
\begin{figure}[ht!]
    \centering
\begin{tikzpicture}[scale=0.45]

\begin{scope}[shift={(-18,0)}]
\twocolor{0/3}{1/1,1/0}
\draw[rounded corners = 0.3mm, very thick, black] (0 cm, 1cm) -- (2 cm, 1 cm) -- (2.5 cm,1.3 cm)--(1.5 cm, 1.3 cm) -- (2 cm,1.6 cm)--(1 cm, 1.6cm)--(0 cm, 1 cm);
\begin{scope}[shift={(1,2.25)}]
\twocolor{0/2}{1/0}
\draw[rounded corners = 0.3mm, very thick, black] (1 cm,1)--(0cm,1cm)--(.5cm,1.3cm)--(1.5cm,1.3cm)--(1,1);
\end{scope}
\begin{scope}[shift={(2,4.25)}]
\twocolor{1/0}{}


\end{scope}

\node at (.75,-4) {$
\begin{array}{ccc}
    0   \\
    0 & 1 \\ 
    1 & 2 & 2
\end{array}
$};

\end{scope}

\node (l1) at (-17,-.3) {};
\node (r1) at (-12,-.3) {};
\draw [->] (l1) to [bend right=30]  node [below, sloped]  (t1) {\small (1)} (r1);

\begin{scope}[shift={(-11.8,0)}]
\reversetwocolor{0/3}{1/1,1/0}
\draw[rounded corners = 0.3mm, very thick, black] (0 cm, 1cm) -- (2 cm, 1 cm) -- (2.5 cm,1.3 cm)--(1.5 cm, 1.3 cm) -- (2 cm,1.6 cm)--(1 cm, 1.6cm)--(0 cm, 1 cm);
\begin{scope}[shift={(1,2.25)}]
\reversetwocolor{0/2}{1/0}
\draw[rounded corners = 0.3mm, very thick, black] (1 cm,1)--(0cm,1cm)--(.5cm,1.3cm)--(1.5cm,1.3cm)--(1,1);
\end{scope}
\begin{scope}[shift={(2,4.25)}]
\reversetwocolor{1/0}{}
\end{scope}

\node at (.75,-4) {$
\begin{array}{ccc}
    1   \\
    1 & 1 \\ 
    0 & 0 & 1
\end{array}
$};

\end{scope}

\node (l2) at (-10.8,-.3) {};
\node (r2) at (-5.6,-.3) {};
\draw [->] (l2) to [bend right=30]  node [below, sloped]  (t2) {\small (2)} (r2);

\begin{scope}[shift={(-4.6,.65)}]
\cube{0}{0} \gcube{1}{0} \gcube{2}{0} 
\cube{-.5}{-.3} \gcube{.5}{-.3}
\gcube{-1}{-.6}
\draw[rounded corners = 0.3mm, very thick, black] (-.5,0.7 cm) -- (.5,0.7cm)--(1,1 cm)--(2,1cm)--(2.5,1.3 cm)  -- (0.5,1.3 cm)--(-.5, 0.7 cm);
\begin{scope}[shift={(1,1.9)}]
\cube{0}{0} \gcube{1}{0}
\gcube{-.5}{-.3}
\draw[rounded corners = 0.3mm, very thick, black] (1,1)--(0cm,1cm)--(0.5cm,1.3cm)--(1.5cm,1.3cm)--(1,1);
\end{scope}
\begin{scope}[shift={(2,3.6)}]
\cube{0}{0}
\end{scope}

\node at (.5,-5) {$
\begin{array}{ccc}
    1 & 1 & 1 \\
    1 & 0 \\ 
    0 & 
\end{array}
$};

\end{scope}

\node (l3) at (-3.6,-.3) {};
\node (r3) at (1.5,-.3) {};
\draw [->] (l3) to [bend right=30]  node [below, sloped]  (t3) {\small (3)} (r3);

\begin{scope}[shift={(4.5,.6)}]
\cube{0}{0}
\cube{-1.5}{-.3} \gcube{-.5}{-.3}
\cube{-3}{-.6} \gcube{-2}{-.6} \gcube{-1}{-.6}
\draw[rounded corners = 0.3mm, very thick, black] (-4,0.4 cm)--(-2, 0.4 cm)--(-1,1 cm)--(-2,1cm)--(-2.5,0.7cm)--(-3.5,0.7 cm) -- (-4,0.4 cm);
\begin{scope}[shift={(-1.5,2)}]
\gcube{0}{0} 
\cube{-1.5}{-.3} \gcube{-.5}{-.3}
\draw[rounded corners = 0.3mm, very thick, black] (-1.5,0.7)--(-2.5cm,0.7cm)--(-2cm,1cm)--(-1cm,1cm)--(-1.5,0.7cm);
\end{scope}
\begin{scope}[shift={(-3,3.6)}]
\gcube{0}{0}
\end{scope}

\node at (-2.25,-5) {$
\begin{array}{ccc}
     &  & 1 \\
    & 1 & 0 \\ 
    2& 0&0  
\end{array}
$};
\end{scope}

\node (l4) at (3,-.3) {};
\node (r4) at (8,-.3) {};
\draw [->] (l4) to [bend right=30]  node [below, sloped]  (t4) {\small (4)} (r4);

\begin{scope}[shift={(8.2,0)}]
\twocolor{2/1}{1/1,0/1}
\draw[rounded corners = 0.3mm, very thick, black] (0 cm, 1cm) -- (2 cm, 1 cm) -- (2.5 cm,1.3 cm)--(1.5 cm, 1.3 cm) -- (2 cm,1.6 cm)--(1 cm, 1.6cm)--(0 cm, 1 cm);
\begin{scope}[shift={(1,2.2)}]
\twocolor{1/1}{1/0}
\draw[rounded corners = 0.3mm, very thick, black] (1 cm,1)--(0cm,1cm)--(.5cm,1.3cm)--(1.5cm,1.3cm)--(1,1);
\end{scope}
\begin{scope}[shift={(2,4.2)}]
\twocolor{1/0}{}
\end{scope}

\node at (.75,-4) {$
\begin{array}{ccc}
    1  \\
    0 & 1 \\ 
    0 & 0 & 2 
\end{array}
$};
\end{scope}

\end{tikzpicture}

    \caption{Example of the omagog to kagog bijection $\psi$. (1) Invert the colors. (2) push the cubes northward along the columns. (3) Tip the stack around the $y$-axis. (4) Reflect through to the plane $y=(n+1)/2$. }
    \label{fig:M2K}
\end{figure}
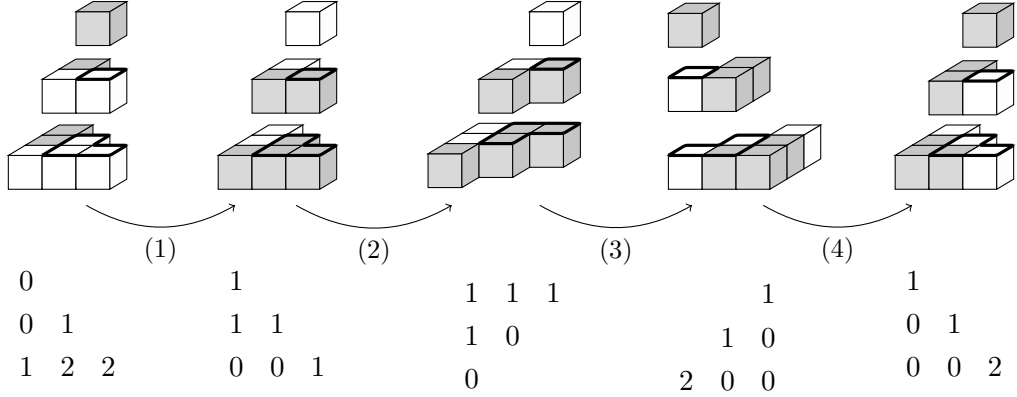


Updating the omagog pyramid inequalities   at every step leads to the following algebraic constraints for some pyramid $\pyr{P}$:
\begin{itemize}
\item[(P1)] $\pyr{P}(i,j,k)$ is defined for $1\leq k \leq j \leq i \leq n-1$. 
\item[(P2)] $\pyr{P}(i,j,k) \geq \pyr{P}(i+1,j,k) $,  
\item[(P3)] $\pyr{P}(i,j,k) \geq \pyr{P}(i,j-1,k-1)$, and 
\item[(P4)] $\pyr{P}(i,j,k) \geq \pyr{P}(i,j,k+1)$. 
\end{itemize}
These pyramid inequalities correspond to the kagog triangle constraints of Definition \ref{def:kagog}, where we must recall that color 1 (white) cubes are present and color 0 (gray) cubes are absent. Condition (P1) ensures that the domain for admissible cubes $(i,j,k)$ is correct and that the height of tower $(i,j)$ is at most $j$, so (K1) and (K2) hold. Condition (P4) states that the cubes adhere to gravity: color 1 blocks must appear below color 0 blocks.
Conditions (P2) and (P4) ensure that the columns are weakly decreasing, so (K3) holds. Conditions (P3) and (P4) ensure that the rows are strictly increasing after the first nonzero entry, so (K4) holds. Indeed, if cube $(i,j-1,k-1)$ is color 1, then $(i,j,k)$ is color 1, so the tower at $(i,j)$ must be taller than the tower at $(i,j-1)$. 
\end{proof}


\subsection{A Catalan Submapping}

In this brief digression, we show that the mapping $\psi:\pyromagog{n} \rightarrow \pyrkagog{n}$
induces a natural bijection  between Catalan subfamilies of  these pyramids. We start by describing two known Catalan families \cite{stanley-catalan}.  Let $\cS_n$ denote the set of nondecreasing sequences $(s_0,s_1, \ldots, s_{n-1})$ where $0 \leq s_i \leq i$ for $0 \leq i \leq n-1$ and $s_i \leq s_{i+1}$ for $0\leq i \leq n-2$. Let $\cC_n$ denote the set of coin pyramids whose bottom row contains $n$ consecutive coins.

Next, we define our associated pyramid families.
Let $\seq{n}' \subset \omagog{n}$ be the set of omagog triangles whose first $n-2$ rows are all zero. 
Let $\coin{n}' \subset \kagog{n}$ be the set of   kagog triangles such that every entry in column $j$ is either $j-1$ or $j$.

\begin{prop}
\label{prop:catalan}
Let $\seq{n}, \coin{n}, \seq{n}'$ and $\coin{n}'$ be the families defined above.
\begin{enumerate}
\item[(a)] There is an elementary bijection $\sigma: \seq{n} \rightarrow \coin{n}.$
\item[(b)] There is an elementary bijection  $\rho: \seq{n} \rightarrow \seq{n}'.$
\item[(c)] There is an elementary bijection  $\tau: \coin{n} \rightarrow \coin{n}'.$
\item[(d)] Restricting the bijection $\psi: \pyromagog{n} \rightarrow \pyrkagog{n}$ from Lemma \ref{lemma:magog-kagog}  to $\pyrseq{n}$  gives a bijection to $\pyrcoin{n}$. Furthermore, this bijection has a natural interpretation in terms of monotone sequences and coin pyramids. Namely,  $\sigma =  \tau^{-1} \circ \psi \circ  \rho$.
\end{enumerate}
\end{prop}



\begin{proof}
Figure \ref{fig:seq-coin} shows the families  $\seq{3}$, $\coin{3}, \seq{3}'$, $\coin{3}'$.  It also shows two  families $\hybrid{3}, \hybrid{3}'$ of hybrid configurations  that are essential in multiple stages of the proof.   

\emph{Proof of (a).}   Our bijection
 joins $\seq{n}$ with $\coin{n}$ via the set 
$\catpath{n}$ of  lattice paths from $(0,0)$ to $(n,n)$ that never travel above the diagonal $y=x$, composing mappings described in \cite{stanley-catalan}.
First, we map sequence $s=(s_0, s_1, \ldots, 
s_{n-1}) \in \seq{n}$ to the lattice path $p \in \catpath{n}$ whose $k$th horizontal step is at height $s_k$. Next, we place gray (missing) coins in each square below $p$, and place white coins in each square above the path $p$, up to and including the squares along the diagonal $y=x$. Let $\hybrid{n}$ denote the hybrid family of configurations of paths and coins, where missing coins are gray.  To complete the mapping $\sigma$, reflect the white coins in the hybrid configuration through $\theta = \pi/8$ to obtain the corresponding coin pyramid.

\begin{figure}[ht]

\centering

\begin{tikzpicture}[scale=.75]

\node at (-1.5,4) {$\seq{3}'$};

\node at (-1.5,2) {$\seq{3}$};

\node at (-1.5,.4) {$\coin{3}$};

\node at (-1.5,-3) {$\hybrid{3}$};

\node at (-1.5,-5.5) {$\hybrid{3}'$};

\node at (-1.5,-8) {$\coin{3}'$};

\begin{scope}

\node at (.75, 4) {$\begin{array}{cc} 0 \\ 0 & 0  \end{array}$};

\node at (.5, 2) {$(0,0,0)$};

\draw (0,0) circle (.25);
\draw (.5,0) circle (.25);
\draw (1,0) circle (.25);
\draw (.25,.45) circle (.25);
\draw (.75,.45) circle (.25);
\draw (.5,.9) circle (.25);

\begin{scope}[shift={(-.625,-4)}]
\foreach \x in {0,.75,1.5,2.25} {
   \draw[gray!50] (\x,0) -- (\x,2.25);
   \draw[gray!50] (0,\x) -- (2.25,\x);
}

\draw[very thick] (0,0) -- (.75,0);
\draw[very thick] (2.25,1.5) -- (2.25,2.25);

\draw[very thick] (.75,0) -- (2.25,0) -- (2.25,1.5);

\draw (.375,.375) circle (.25);
\draw (1.125,.375) circle (.25);
\draw (1.895,.375) circle (.25);
\draw (1.125,1.125) circle (.25);
\draw (1.895,1.125) circle (.25);
\draw (1.895,1.895) circle (.25);

\draw (1.125,-1.25) circle (.25);
\draw (1.125,-2) circle (.25);
\draw (1.895,-2) circle (.25);

\end{scope}


\node at (1,-8) {$\begin{array}{cc} 1 \\ 1 &2 \end{array}$};

\end{scope}

\begin{scope}[shift={(4,0)}]

\node at (.75, 4) {$\begin{array}{cc} 0 \\ 0 & 1  \end{array}$};

\node at (.5, 2) {$(0,0,1)$};

\draw (0,0) circle (.25);
\draw (.5,0) circle (.25);
\draw (1,0) circle (.25);
\draw (.25,.45) circle (.25);
\draw (.75,.45) circle (.25);

\begin{scope}[shift={(-.625,-4)}]
\foreach \x in {0,.75,1.5,2.25} {
   \draw[gray!50] (\x,0) -- (\x,2.25);
   \draw[gray!50] (0,\x) -- (2.25,\x);
}

\draw[very thick] (0,0) -- (.75,0);
\draw[very thick] (2.25,1.5) -- (2.25,2.25);

\draw[very thick] (.75,0) -- (1.5,0) -- (1.5, .75) -- (2.25,.75) -- (2.25,1.5);

\draw (.375,.375) circle (.25);
\draw (1.125,.375) circle (.25);
\draw[gray, fill=gray!50] (1.895,.375) circle (.25);
\draw (1.125,1.125) circle (.25);
\draw (1.895,1.125) circle (.25);
\draw (1.895,1.895) circle (.25);

\draw (1.125,-1.25) circle (.25);
\draw[gray, fill=gray!50]  (1.125,-2) circle (.25);
\draw (1.895,-2) circle (.25);

\end{scope}


\node at (1,-8) {$\begin{array}{cc} 1 \\ 0 &2 \end{array}$};

\end{scope}

\begin{scope}[shift={(8,0)}]

\node at (.75, 4) {$\begin{array}{cc} 0 \\ 1 & 1  \end{array}$};

\node at (.5, 2) {$(0,1,1)$};

\draw (0,0) circle (.25);
\draw (.5,0) circle (.25);
\draw (1,0) circle (.25);
\draw (.75,.45) circle (.25);

\begin{scope}[shift={(-.625,-4)}]
\foreach \x in {0,.75,1.5,2.25} {
   \draw[gray!50] (\x,0) -- (\x,2.25);
   \draw[gray!50] (0,\x) -- (2.25,\x);
}

\draw[very thick] (0,0) -- (.75,0);
\draw[very thick] (2.25,1.5) -- (2.25,2.25);

\draw[very thick] (.75,0) -- (.75,.75) -- (1.5, .75) -- (2.25,.75) -- (2.25,1.5);

\draw (.375,.375) circle (.25);
\draw[gray, fill=gray!50] (1.125,.375) circle (.25);
\draw[gray, fill=gray!50] (1.895,.375) circle (.25);
\draw (1.125,1.125) circle (.25);
\draw (1.895,1.125) circle (.25);
\draw (1.895,1.895) circle (.25);

\draw (1.125,-1.25) circle (.25);
\draw[gray, fill=gray!50]  (1.125,-2) circle (.25);
\draw[gray, fill=gray!50]  (1.895,-2) circle (.25);

\end{scope}


\node at (1,-8) {$\begin{array}{cc} 1 \\ 0 &1 \end{array}$};

\end{scope}

\begin{scope}[shift={(12,0)}]

\node at (.75, 4) {$\begin{array}{cc} 0 \\ 0 & 2  \end{array}$};

\node at (.5, 2) {$(0,0,2)$};

\draw (0,0) circle (.25);
\draw (.5,0) circle (.25);
\draw (1,0) circle (.25);
\draw (.25,.45) circle (.25);

\begin{scope}[shift={(-.625,-4)}]
\foreach \x in {0,.75,1.5,2.25} {
   \draw[gray!50] (\x,0) -- (\x,2.25);
   \draw[gray!50] (0,\x) -- (2.25,\x);
}

\draw[very thick] (0,0) -- (.75,0);
\draw[very thick] (2.25,1.5) -- (2.25,2.25);

\draw[very thick] (.75,0) -- (1.5,0) -- (1.5, 1.5) -- (2.25,1.5);

\draw (.375,.375) circle (.25);
\draw (1.125,.375) circle (.25);
\draw[gray, fill=gray!50] (1.895,.375) circle (.25);
\draw (1.125,1.125) circle (.25);
\draw[gray, fill=gray!50] (1.895,1.125) circle (.25);
\draw (1.895,1.895) circle (.25);

\draw[gray, fill=gray!50]  (1.125,-1.25) circle (.25);
\draw[gray, fill=gray!50]  (1.125,-2) circle (.25);
\draw (1.895,-2) circle (.25);

\end{scope}


\node at (1,-8) {$\begin{array}{cc} 0 \\ 0 &2 \end{array}$};

\end{scope}

\begin{scope}[shift={(16,0)}]

\node at (.75, 4) {$\begin{array}{cc} 0 \\ 1 & 2  \end{array}$};

\node at (.5, 2) {$(0,1,2)$};

\draw (0,0) circle (.25);
\draw (.5,0) circle (.25);
\draw (1,0) circle (.25);

\begin{scope}[shift={(-.625,-4)}]
\foreach \x in {0,.75,1.5,2.25} {
   \draw[gray!50] (\x,0) -- (\x,2.25);
   \draw[gray!50] (0,\x) -- (2.25,\x);
}

\draw[very thick] (0,0) -- (.75,0);
\draw[very thick] (2.25,1.5) -- (2.25,2.25);

\draw[very thick] (.75,0) -- (0.75,.75) -- (1.5, .75) -- (1.5, 1.5) -- (2.25,1.5);

\draw (.375,.375) circle (.25);
\draw[gray, fill=gray!50] (1.125,.375) circle (.25);
\draw[gray, fill=gray!50] (1.895,.375) circle (.25);
\draw (1.125,1.125) circle (.25);
\draw[gray, fill=gray!50] (1.895,1.125) circle (.25);
\draw (1.895,1.895) circle (.25);

\draw[gray, fill=gray!50]  (1.125,-1.25) circle (.25);
\draw[gray, fill=gray!50]  (1.125,-2) circle (.25);
\draw[gray, fill=gray!50]  (1.895,-2) circle (.25);

\end{scope}


\node at (1,-8) {$\begin{array}{cc} 0 \\ 0 &1 \end{array}$};

\end{scope}

\end{tikzpicture}

\caption{Six Catalan families used in the proof of Proposition \ref{prop:catalan}. The family $\seq{3}$ of monotone sequences $(s_0,s_1,s_2)$ maps simply to the subfamily $\seq{3}'$ of omagog pyramids whose first row is zero.  The family $\coin{3}$ of coin pyramids is in bijection with $\seq{3}$ via the hybrid family $\hybrid{3}$ consisting of lattice paths and coins, where $s_k$ is the height of the horizontal step starting at $x=k$. We map $\coin{3}$ to the subfamily $\coin{3}'$ of kagog pyramids via  family  $\hybrid{3}'$, the mirror image of the non-diagonal coins of $\hybrid{3}.$ }

\label{fig:seq-coin}

\end{figure}

\emph{Proof of (b).}
The monotone sequences $\seq{n}$ map quite  simply  to $\seq{n}'$.  The sequence $s \in \seq{n}$  maps  to the omagog triangle in $\omagog{n}$ whose final row is $(s_1, \dots, s_{n-1})$, and whose other  rows are all-zero. This mapping is clearly a bijection.

\emph{Proof of (c).}
We map coin pyramids $\coin{n}$ to the triangles in $\coin{n}'$ via the hybrid configurations in $\hybrid{n}$.   After mapping  a coin pyramid to its hybrid configuration in $\hybrid{n}$, we
ignore the white coins on the diagonal  (which correspond to the fixed base of the coin pyramid), and reflect the remaining coins across the vertical axis  to get a triangular array of the appropriate shape. Let $\hybrid{n}'$ denote the resulting family of triangular arrays of two-colored coins. Replace each white coin with a 1 and each gray coin with a 0. Finally, add $j-1$ to the entries in column $j$ for $1 \leq j \leq n-1$. The result is a kagog triangle in $\coin{n}'$.  This invertible mapping is a bijection.

\emph{Proof of (d).}
First, we show that the bijection $\psi: \pyromagog{n} \rightarrow \pyrkagog{n}$ maps
$\pyrseq{n}$ to $\pyrcoin{n}$. All of the white (present) blocks of   $\pyr{S} \in \pyrseq{n}$ are in row $n-1$. Let $\pyr{K} = \psi(\pyr{S})$ where $\psi$ is the mapping in the proof of Lemma \ref{lemma:magog-kagog}. 
Recall that in this mapping, the  block $\pyr{S}(i,j,k)$ flips colors and moves to  $\pyr{K}(n-k,n-j,i-j+1)$.
In particular, the gray block $\pyr{S}(n-2,j,k)$ maps to the white block $\pyr{K}(n-k,n-j,n-j-1)$. This proves that every tower in column $\ell = n-j$ has height at least $\ell -1 = n-j-1$;  in other words,  $\psi$ bijects $\pyrseq{n}$ to $\pyrcoin{n}$. 

It remains to show that the mapping $\psi$ corresponds to the mapping $\sigma:  \seq{n} \rightarrow \coin{n}$.  The key  is to  take a bird's eye view of a kagog pyramid $\pyr{K} \in \pyrcoin{n}$.  This view only shows the topmost blocks; this is sufficient, since the blocks in the lower layers are all white. 
We will see that the coin colors of $h' \in \hybrid{n}'$ correspond to the block colors of the top layer of a unique $\pyr{K} = \psi(\pyr{S})$.
Keeping this intuition in mind, we conclude the proof.
After mapping, the  block $\pyr{S}(n-1,j,k)$ flips color and maps to the top-layer block  $\pyr{K}(n-k,n-j,n-j)$.
Suppose that $\pyr{S}(n-1,j,k)$ is white for $1 \leq k \leq \ell$ and gray for $\ell+1 \leq k \leq j$.
This means that $\pyr{K}(n-k, n-j, n-j)$ is gray for $1 \leq k \leq \ell$  and white for $\ell+1 \leq k \leq j$.  
In other words,  
$\pyr{K}(k', j', j')$ is gray for $n - \ell  \leq k' \leq n-1$  and white for $j'  \leq k' \leq n - \ell -1$.  
The bird's eye view of the pyramids of $\pyrkagog{n}$  bijects to the  hybrid configurations of $\hybrid{n}'$, where we replace the blocks with coins.
\end{proof}

%% file: gog-results.tex
\section{The Gog Triangle Involution}

\label{sec:gog}

Given the success of the two-color cube pyramid view of magog triangles, we conclude the paper by investigating two-color cube pyramids of gog triangles. This culminates in our proof of Theorem \ref{thm:gog}.

\begin{definition} 
\label{def:gog}
A gog triangle   $G$ of size $n$ is a triangular array of positive integers $G(i,j)$ such that
\begin{enumerate}
\item[(G1)] $1 \leq j \leq i \leq n$, so the array is triangular;
\item[(G2)] $1 \leq G(i,j) \leq n-i+j$, so entry $j$ in  row $i$ is at most $n-i+j$;
\item[(G3)] $G(i,j) < G(i,j+1)$, so rows are strictly increasing;
\item[(G4)] $G(i,j) \geq G(i+1,j)$, so columns are weakly decreasing; and 
\item[(G5)] $G(i,j) \leq G(i+1,j+1)$, so diagonals are weakly increasing.
\end{enumerate}
We use  $\gog{n}$ to denote the set of gog triangles of size $n$. 
\end{definition}

Note that we use (G2) in place of the standard condition that $G(n,j)=j$ for $1\leq j \leq n$; an elementary argument shows that they are interchangeable. Condition (G2), which applies to all rows of the triangle, is better suited for our two-colored cube pyramid argument.

The gog triangles of $\gog{3}$ are
\begin{equation}
\label{eqn:gog3}
\begin{array}{ccc}
1& &  \\
1& 2& \\
1 & 2 & 3
\end{array}
\quad
\begin{array}{ccc}
1& &  \\
1& 3& \\
1 & 2 & 3
\end{array}
\quad
\begin{array}{ccc}
2& &  \\
1& 2& \\
1 & 2 & 3
\end{array}
\quad
\begin{array}{ccc}
2& &  \\
1& 3& \\
1 & 2 & 3
\end{array}
\quad
\begin{array}{ccc}
2& &  \\
2& 3& \\
1 & 2 & 3
\end{array}
\quad
\begin{array}{ccc}
3& &  \\
1& 3& \\
1 & 2 & 3
\end{array}
\quad
\begin{array}{ccc}
3& &  \\
2& 3& \\
1 & 2 & 3
\end{array}
\end{equation}
where once again, we have chosen to left-justify the triangular arrays.

Gog triangles (also known as monotone triangles) are in bijection with alternating sign matrices.  An ASM is a $n \times n$ matrix of 0's, 1's and $-1$'s such that each or column sums to 1, and the nonzero entries in each row or column alternate in sign. The seven $3 \times 3$ ASMs are
\begin{equation} \small \left[ \hspace{-3 pt} \begin{array}{ccc}1 & 0 & 0 \\0 & 1 & 0 \\0 & 0 & 1\end{array}\hspace{-3 pt}\right] \:
\left[\hspace{-3 pt}\begin{array}{ccc}1 & 0 & 0 \\0 & 0 & 1 \\0 & 1 & 0\end{array}\hspace{-3 pt}\right] \:
\left[\hspace{-3 pt}\begin{array}{ccc}0 & 1 & 0 \\1 & 0 & 0 \\0 & 0 & 1\end{array}\hspace{-3 pt}\right] \:
\left[\hspace{-3 pt}\begin{array}{ccc}0 & 1 & 0 \\1 & -1 & 1 \\0 & 1 & 0\end{array}\hspace{-3 pt}\right] \:
\left[\hspace{-3 pt}\begin{array}{ccc}0 & 1 & 0 \\0 & 0 & 1 \\1 & 0 & 0\end{array}\hspace{-3 pt}\right] \:
\left[\hspace{-3 pt}\begin{array}{ccc}0 & 0 & 1 \\1 & 0 & 0 \\0 & 1 & 0\end{array}\hspace{-3 pt}\right] \:
\left[\hspace{-3 pt}\begin{array}{ccc}0 & 0 & 1 \\0 & 1 & 0 \\1 & 0 & 0\end{array}\hspace{-3 pt}\right].
\label{eqn:asm3}
\end{equation}
We have listed the seven gog triangles in equation \eqref{eqn:gog3} in the same order as their corresponding $3 \times 3$ ASMs in equation \eqref{eqn:asm3}. The bijective mapping of Mills et al.~\cite{mills} between ASMs and gog triangles is simple to describe: the $j$th row of the gog triangle records the locations of the 1's in the vector obtained by adding the first $j$ rows of the corresponding ASM .




We now prove Theorem \ref{thm:gog}: 
there is a natural gog triangle involution $f: \gog{n} \rightarrow \gog{n}$  that corresponds to both (1) an affine transformation of two-color pyramids, and (2) reversing the order of the rows of the corresponding ASM.

Analogous to Section \ref{sec:magog-kagog}, we start by  defining ogog triangles. The minimum gog triangle has $\mini{G}(i,j)=j$ for all entries $(i,j)$, while the maximum gog triangle  has $\maxi{G}(i,j)=n-i+j$. For every gog triangle, we construct its ogog counterpart by subtracting the minimum gog triangle. The last row in every gog triangle is always $[1 \; 2 \; \cdots \; n]$ since it has length $n$ and is strictly increasing. As such, every ogog triangle has a final row of zeros, which we omit from ogog triangle.

\begin{definition}
An ogog triangle $\oo{G}$ of index $n$ is a triangular array of nonnegative integers $\oo{G}(i,j)$ such that
\begin{enumerate}
\item[(OG1)] $1 \leq j \leq i \leq n-1$, so the array is triangular;
\item[(OG2)] $0 \leq \oo{G}(i,j) \leq n-i$, so values in row $i$ are at most $n-i$;
\item[(OG3)] $\oo{G}(i,j) \leq \oo{G}(i,j+1)$, so rows are weakly increasing;
\item[(OG4)] $\oo{G}(i,j) \geq \oo{G}(i+1,j)$, so columns are weakly decreasing; and
\item[(OG5)] $\oo{G}(i,j) \leq \oo{G}(i+1,j+1) + 1$, so diagonals cannot decrease by more than 1.
\end{enumerate}
We use $\ogog{n}$ to denote the set of all ogog triangles of index $n$.
\end{definition}
For example, the ogog triangles of $\ogog{3}$ are
$$
\begin{array}{ccc}
0& &  \\
0& 0& \\
\end{array}
\quad
\begin{array}{ccc}
0& &  \\
0& 1& \\
\end{array}
\quad
\begin{array}{ccc}
1& &  \\
0& 0& \\
\end{array}
\quad
\begin{array}{ccc}
1& &  \\
0& 1&  \\
\end{array}
\quad
\begin{array}{ccc}
1& &  \\
1& 1&  \\
\end{array}
\quad
\begin{array}{ccc}
2& &  \\
0& 1& \\
\end{array}
\quad
\begin{array}{ccc}
2& &  \\
1& 1& \\
\end{array},
$$
where these ogog triangles are ordered so that they biject to the gog triangles in equation \eqref{eqn:gog3}.
As constructed, the color 1 (white) cubes are present in the ogog triangle, while the color 0 (gray)  cubes are absent. Our next lemma states that the gray  cubes also represent a gog triangle.

\begin{lemma}
\label{lem:ogog-involution}
Let $\oo{G}$ be an ogog triangle and let $\pyro{G}$ be its two-color cube pyramid representation. The color 0 cubes of $\pyro{G}$ are an affine transformation of another ogog triangle.
\end{lemma}

Note that this correspondence is  an involution on the set of ogog triangles: swapping the colors twice leads us back to the original two-coloring of the cube pyramid.

\begin{proof}
Similar to the proof of Lemma \ref{lemma:magog-kagog}, we describe ogog pyramids via a set of inequalities, perform a multistep transformation and then check that the resulting inequalities also describe the set of ogog pyramids. 
The inequalities for ogog pyramids are:
\begin{itemize}
\item $\pyr{\oo{G}}(i,j,k)$ is defined $1 \leq j \leq i \leq n-1$: length of row $i$ is at most $i$,
\item $\pyr{\oo{G}}(i,j,k)$ is defined  $1 \leq k \leq n-i$: height of row $i$ is at most $n-i$,
\item if $j < i$, then $\pyr{\oo{G}}(i,j+1,k)\geq \pyr{\oo{G}}(i,j,k)$: rows are weakly increasing, 
\item if $i>1$, then $\pyr{\oo{G}}(i-1,j,k) \geq \pyr{\oo{G}}(i,j,k)$: columns are weakly decreasing
\item  if $i < n$, then $\pyr{\oo{G}}(i+1,j+1,k-1) \geq  \pyr{\oo{G}}(i,j,k)$: diagonals cannot decrease by more than 1, and
\item if  $k > 1$, then $\pyr{\oo{G}}(i,j,k-1) \geq \pyr{\oo{G}}(i,j,k)$: the present cubes obey gravity.
\end{itemize}
The three-step mapping $\phi$ is:
\begin{itemize}
\item {\bf Step One:} Invert the colors, or exchange color 1 for color 0 and vice versa. This reverses the inequalities.

\item {\bf Step Two:} Perform a quarter rotation of $\R^3$ about the $x$-axis. This moves the cube $(i,j,k)$ to position $(i, -k, j)$. This tips the two-color cube pyramid onto its side. 

\item 
{\bf Step Three:} Rotate by $\pi$ around the $z$-axis  and then translate by $(n,0,0)$. This moves cube $(i,j,k)$ to $(n-i, -j, k)$. 
\end{itemize}
After composing these three steps, cube $(i,j,k)$ switches color and moves to $(n-i,k,j)$. Figure \ref{fig:G2G} exemplifies the mapping $\phi$ for an ogog pyramid from $\pyrogog{4}$. 

\begin{figure}[ht!]
    \centering
\begin{tikzpicture}[scale=0.45]

\begin{scope}[shift={(-18,0)}]
\begin{scope}[shift={(0,0)}]
\twocolor{2/1}{1/1,0/1}
\draw[rounded corners = 0.3mm, very thick, black] (-.5 cm, 1.3cm) -- (1.5 cm, 1.3 cm) -- (2 cm,1.6 cm)--(1 cm, 1.6 cm) -- (1.5 cm,1.9 cm)--(0.5 cm, 1.9cm)--(-0.5 cm, 1.3 cm);
\end{scope}
\begin{scope}[shift={(.5,2.5)}]
\twocolor{1/1}{0/1}
\draw[rounded corners = 0.3mm, very thick, black] (0.5,1.3)--(-0.5cm,1.3cm)--(0cm,1.6cm)--(1cm,1.6cm)--(0.5,1.3);
\end{scope}
\begin{scope}[shift={(1,4.5)}]
\twocolor{1/0}{}
\end{scope}

\node at (.75,-4) {$
\begin{array}{ccc}
    2   \\
    0 & 2 \\ 
    0 & 0 & 1
\end{array}
$};

\end{scope}

\node (l1) at (-17,-.3) {};
\node (r1) at (-12,-.3) {};
\draw [->] (l1) to [bend right=30]  node [below, sloped]  (t1) {\small (1)} (r1);

\begin{scope}[shift={(-11.8,0)}]

\begin{scope}[shift={(0,0)}]
\reversetwocolor{2/1}{1/1,0/1}
\draw[rounded corners = 0.3mm, very thick, black] (-.5 cm, 1.3cm) -- (1.5 cm, 1.3 cm) -- (2 cm,1.6 cm)--(1 cm, 1.6 cm) -- (1.5 cm,1.9 cm)--(0.5 cm, 1.9cm)--(-0.5 cm, 1.3 cm);
\end{scope}
\begin{scope}[shift={(.5,2.5)}]
\reversetwocolor{1/1}{0/1}
\draw[rounded corners = 0.3mm, very thick, black] (0.5,1.3)--(-0.5cm,1.3cm)--(0cm,1.6cm)--(1cm,1.6cm)--(0.5,1.3);
\end{scope}
\begin{scope}[shift={(1,4.5)}]
\reversetwocolor{1/0}{}
\end{scope}

\node at (.75,-4) {$
\begin{array}{ccc}
    1   \\
    2 & 0 \\ 
    1 & 1 & 0
\end{array}
$};

\end{scope}

\node (l2) at (-10.8,-.3) {};
\node (r2) at (-5.6,-.3) {};
\draw [->] (l2) to [bend right=30]  node [below, sloped]  (t2) {\small (2)} (r2);

\begin{scope}[shift={(-6,0)}]

\begin{scope}[shift={(1,.66)}]
\reversetwocolor{1/2}{}
\end{scope}
\begin{scope}[shift={(1.5,.33)}]
\reversetwocolor{2/0}{}
\end{scope}
\begin{scope}[shift={(2,0)}]
\reversetwocolor{1/0}{}
\end{scope}
\draw[rounded corners = 0.3mm, very thick, black] (1,1cm)--(2cm,1cm)--(3,1.6cm)--(1,1.6cm)--(0.5cm,1.3cm)--(1.5cm,1.3cm)--(1,1cm);

\begin{scope}[shift={(0,2.5)}]
\begin{scope}[shift={(1.5,.33)}]
\reversetwocolor{0/2}{}
\end{scope}
\begin{scope}[shift={(2,0)}]
\reversetwocolor{1/0}{}
\draw[rounded corners = 0.3mm, very thick, black] (0,1)--(-1cm,1cm)--(-.5cm,1.3cm)--(0.5cm,1.3cm)--(0,1);
\end{scope}
\end{scope}

\begin{scope}[shift={(2.25,4.5)}]
\reversetwocolor{0/1}{}
\end{scope}

\node at (1.5,-4) {$
\begin{array}{ccc}
    1 & 0 & 0 \\
    & 1 & 1 \\ 
    & & 2
\end{array}
$};

\end{scope}

\node (l3) at (-3.6,-.3) {};
\node (r3) at (1.5,-.3) {};
\draw [->] (l3) to [bend right=30]  node [below, sloped]  (t3) {\small (3)} (r3);

\begin{scope}[shift={(2,0)}]

\begin{scope}[shift={(0,0)}]
\twocolor{2/1}{0/2,0/1}
\draw[rounded corners = 0.3mm, very thick, black] (-.5 cm, 1.3cm) -- (1.5 cm, 1.3 cm) -- (2 cm,1.6 cm)--(1 cm, 1.6 cm) -- (1.5 cm,1.9 cm)--(0.5 cm, 1.9cm)--(-0.5 cm, 1.3 cm);
\end{scope}
\begin{scope}[shift={(.5,2.5)}]
\twocolor{2/0}{0/1}
\draw[rounded corners = 0.3mm, very thick, black] (0.5,1.3)--(-0.5cm,1.3cm)--(0cm,1.6cm)--(1cm,1.6cm)--(0.5,1.3);
\end{scope}
\begin{scope}[shift={(1,4.5)}]
\twocolor{1/0}{}
\end{scope}

\node at (1,-4) {$
\begin{array}{ccc}
    2 &  &  \\
    1 & 1 &  \\ 
    0& 0&1  
\end{array}
$};
\end{scope}

\end{tikzpicture}

    \caption{Example of the ogog to ogog bijection $\phi$. (1) Invert the colors. (2) Tip the stack around the $x$-axis by a quarter turn. (3)  Rotate one half turn about the $z$-axis, then translate by $(n,0,0)$. }
    \label{fig:G2G}
\end{figure}
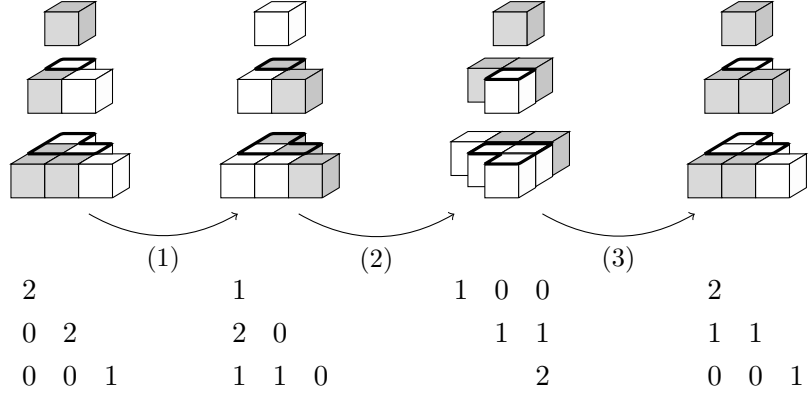

Careful algebra shows that the resulting constraints are a permutation of the algebraic inequalities  for an ogog cube pyramid. As such, this mapping takes one gog triangle to another gog triangle. This affine mapping is an involution, so it is bijective.
\end{proof}


The ogog pyramids are in bijection with gog triangles, and hence also in bijection with alternating sign matrices. Our next corollary shows that the  involution $\phi$ of Lemma \ref{lem:ogog-involution}  reverses the rows of the associated ASM.


\begin{cor}
\label{cor:gog}
Let $\phi$ be the ogog pyramid involution of Lemma \ref{lem:ogog-involution}.
Let $A$ be an $n \times n$ alternating sign matrix corresponding to ogog triangle $\oo{G}$ with two-color cube pyramid $\pyr{\oo{G}}$. Let $\pyro{H} = \phi(\pyr{\oo{G}})$ and let $\oo{H}$ be the associated ogog triangle.  Then $\oo{H}$ is the ogog triangle corresponding to processing the rows of $A$ in reverse order.
\end{cor}

Figure \ref{fig:gogmap} shows an example of how the alternating sign matrices $A$ and $B$ corresponding to ogog pyramids $\pyr{\oo{G}}$ and $\pyr{\oo{H}} = \phi(\pyr{\oo{G}})$ are the row reversals of one another.

\begin{figure}[ht!]
    \centering
\begin{tikzpicture}

\foreach \x in {-1.5, -4.5, -7.5, -10}
{
\node at (0, \x) {$\updownarrow$};
\node at (5, \x) {$\updownarrow$};
}

\node at (2.5,0) {$\longleftrightarrow$};
\node at (2.5,0.05) {\small{\begin{tabular}{c} row \\ reversal \end{tabular}}};
\node at (2.5,-12) {$\longleftrightarrow$};
\node at (2.5,-11.75) {$\phi$};

\node at (10,0) {ASM};
\node at (10,-3) {ASM partial sums};
\node at (10,-5.9) {gog triangle};
\node at (10,-8.75) {ogog triangle};
\node at (10,-11.75) {two-color cube pyramid};

\node at (-2,0) {$A$};

\node at (0,0)
{\small{
$
\left[
\begin{array}{cccc}
0 & 0 & 1 & 0 \\
1 & 0 & -1 & 1 \\
0 & 1 & 0 & 0 \\
0 & 0 & 1 & 0 
\end{array}
\right]
$
}};


\node at (-2,-3) {$A'$};

\node at (0,-3)
{\small{
$
\small
\left[
\begin{array}{cccc}
0 & 0 & 1 & 0 \\
1 & 0 & 0 & 1 \\
1 & 1 & 0 & 1 \\
1 & 1 & 1 & 1 
\end{array}
\right]
$
}};

\node at (0,-5.9)
{\small{
$
\begin{array}{cccc}
3 \\
1 & 4 \\
1 & 2 & 4 \\
1 & 2 & 3 & 4
\end{array}
$
}};
\node at (-2, -5.9) {$G$};

\node at (0,-8.75)
{\small{
$
\begin{array}{cccc}
2 \\
0 & 2 \\
0 & 0 & 1 & \phantom{0}
\end{array}
$
}};
\node at (-2, -8.85) {$\oo{G}$};

\node at (-1.75, -11.75) {$\pyr{\oo{G}}$};
\node at (0,-11.75)
{
\begin{tikzpicture}[scale=0.4]


\begin{scope}[shift={(0,0)}]
\twocolor{2/1}{1/1,0/1}
\draw[rounded corners = 0.3mm, very thick, black] (-.5 cm, 1.3cm) -- (1.5 cm, 1.3 cm) -- (2 cm,1.6 cm)--(1 cm, 1.6 cm) -- (1.5 cm,1.9 cm)--(0.5 cm, 1.9cm)--(-0.5 cm, 1.3 cm);
\end{scope}
\begin{scope}[shift={(.5,2.25)}]
\twocolor{1/1}{0/1}
\draw[rounded corners = 0.3mm, very thick, black] (0.5,1.3)--(-0.5cm,1.3cm)--(0cm,1.6cm)--(1cm,1.6cm)--(0.5,1.3);
\end{scope}
\begin{scope}[shift={(1,4.25)}]
\twocolor{1/0}{}
\end{scope}
\end{tikzpicture}
};


\node at (7,0) {$B$};

\node at (5,0)
{\small{
$
\left[
\begin{array}{cccc}
0 & 0 & 1 & 0 \\
0 & 1 & 0 & 0 \\
1 & 0 & -1 & 1 \\
0 & 0 & 1 & 0 
\end{array}
\right]
$
}};


\node at (7,-3) {$B'$};

\node at (5,-3)
{\small{
$
\left[
\begin{array}{cccc}
0 & 0 & 1 & 0 \\
0 & 1& 1 & 0 \\
1 & 1 & 0 & 1 \\
1 & 1 & 1 & 1 
\end{array}
\right]
$
}};

\node at (5,-5.9)
{\small{
$
\begin{array}{cccc}
3 \\
2 & 3 \\
1 & 2 & 4 \\
1 & 2 & 3 & 4
\end{array}
$
}};
\node at (7, -5.9) {$H$};

\node at (5,-8.75)
{\small{
$
\begin{array}{ccc}
2 \\
1 & 1 \\
0 & 0 & 1 
\end{array}
$
}};
\node at (7, -8.85) {$\oo{H}$};

\node at (6.75, -11.75) {$\pyr{\oo{H}}$};

\node at (5,-11.75)
{
\begin{tikzpicture}[scale=0.4]


\begin{scope}[shift={(0,0)}]
\twocolor{2/1}{0/2,0/1}
\draw[rounded corners = 0.3mm, very thick, black] (-.5 cm, 1.3cm) -- (1.5 cm, 1.3 cm) -- (2 cm,1.6 cm)--(1 cm, 1.6 cm) -- (1.5 cm,1.9 cm)--(0.5 cm, 1.9cm)--(-0.5 cm, 1.3 cm);
\end{scope}
\begin{scope}[shift={(.5,2.25)}]
\twocolor{2/0}{0/1}
\draw[rounded corners = 0.3mm, very thick, black] (0.5,1.3)--(-0.5cm,1.3cm)--(0cm,1.6cm)--(1cm,1.6cm)--(0.5,1.3);
\end{scope}
\begin{scope}[shift={(1,4.25)}]
\twocolor{1/0}{}
\end{scope}
\end{tikzpicture}
};

\end{tikzpicture}

    \caption{The alternating sign matrices $A$ and $B$ corresponding to the two-color cube pyramids $\pyr{\oo{G}}$ and $\pyr{\oo{H}} = \phi(\pyr{\oo{G}})$ are the row reversals of one another.}
    \label{fig:gogmap}
\end{figure}
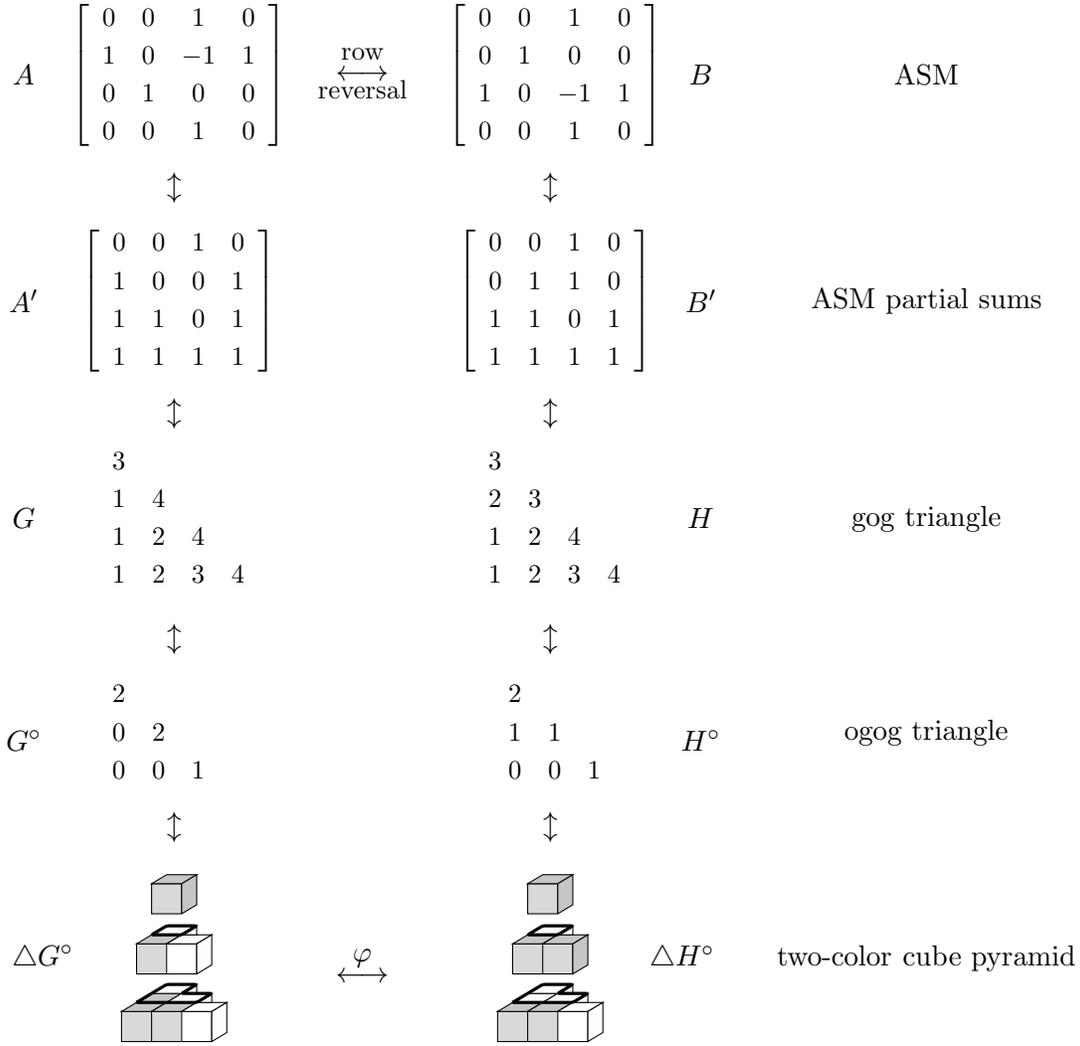

\begin{proof}
Starting with the alternating sign matrix $A$, we obtain the ogog triangle $\oo{G}$ as follows. First, we create  the matrix $A'$ whose $i$th row is the sum of the first $i$ rows of $A$. This is a 0-1 matrix whose $i$th row contains exactly $i$ ones. We convert $A'$ into a gog triangle $G$ by reporting the indices of the ones in each row. We then set $\oo{G} = G - \mini{G}$, which corresponds to subtracting $[  1, 2 , \ldots, i ]$ from row $i$ of $G$ for $1 \leq i \leq n$ and then deleting the final row (which is all-zero).

Let $A_i$ denote the $i$th row of $A$ and let  $A'_i = \sum_{k=1}^k A_i$ denote the  $i$th row of the partial sum matrix $A'$. 
Let $1 \leq a'_1 < a'_2 < \cdots < a'_i \leq n$ denote the locations of the ones in row $A'_i.$  Then $G(i,j)=a'_j$, or equivalently $[a'_1, a'_2, \cdots, a'_i]$ is the $i$th row of the gog triangle $G$. The entries satisfy
$$
\begin{array}{ccccccc}
1 &\leq& a'_1 &\leq& n-i+1,  \\
a'_{j-1} &<& a'_j &\leq& n-i+j, &&  2 \leq j \leq i.
\end{array}
$$
Row $i$ of ogog triangle $\oo{G}$ is
\begin{equation}
\label{eqn:ogog-row}
[a'_1-1, a'_2-2 \cdots, a'_i - i].
\end{equation} 

We start with row $n-1$ of our triangle, as it is the simplest row to comprehend. Row $n-1$ of gog triangle $G$ is  $[a'_1, a'_2, \ldots , a'_{n-1}]$, which is missing a  single number $\ell \in [n]$, namely the location $\ell$ of  the unique one in row $n$ of $A$. By equation \eqref{eqn:ogog-row}, the corresponding ogog row consists of $\ell-1$ zeros followed by $n-\ell$ ones. 

Consider this row in the context of the two-color ogog pyramid $\pyr{\oo{G}}$ and its image
$\pyr{\oo{H}} = \phi(\pyr{\oo{G}})$.
Row $n-1$ of pyramid $\pyr{\oo{G}}$ has height 1. It contains $\ell-1$ cubes of color 0, followed by $n-\ell$ cubes of color 1. 
After transformation $\phi$, the cube $(n-1,j,1)$ switches color and moves to $(1,1,j)$. So   $\pyr{\oo{H}}$ has a tower of blocks at $(1,1)$ of height $n-1$, with $\ell-1$ cubes of color 1 below $n-\ell$ cubes of color 0. It follows that ogog triangle $\oo{H}$ has
$\oo{H}(1,1) = \ell-1$, and thus the corresponding gog triangle $H$ has  $H(1,1)=\ell$. This confirms that the first row of gog triangle $H$ corresponds to the last row of matrix $A$, as desired.

We now handle a generic row $i$ of ogog triangle $\oo{G}$; Figure \ref{fig:gog-wall} shows an example. The entries of row $i$ are a weakly increasing list of length $i$, drawn from $ \{ 0, 1, \ldots, n-i \}$. Let $0 \leq s_m \leq i $ be the number of consecutive $m$'s in this list, so that
$$\sum_{m=0}^{n-i} s_m = i.$$ 
In the corresponding gog triangle $G$, row $i$ is missing the integers 
\begin{equation} 
\label{eqn:sum-of-missing}
1 + p +  \sum_{k=0}^{p} s_k \quad \mbox{where} \quad 0 \leq p \leq n-i-1.
\end{equation}
Let us pause to make some key observations.
The missing integers in row $i$ of  $G$ are precisely the locations of the zeros in the partial sum $A'_i = \sum_{\ell=1}^{i} A_{\ell}$.
Since the sum of all the rows yeilds the all-ones vector, these are also the locations of the ones in the partial sum $\sum_{\ell=i+1}^{n} A_{\ell}$. Of course,  summing  the last $n-i$ rows of $A$ is the same as  summing the first $n-i$ rows of the row reversal of $A$.  

Next, we translate our observations into statements about two-color pyramids. When we convert ogog triangle $\oo{G}$ into pyramid $\pyro{G}$, row $i$ of $\oo{G}$ maps to the $i \times (n-i)$ wall of cubes 
$$
\pyro{G_i} = \{ (i,j,k) : 1 \leq j \leq i \mbox{ and } 1 \leq k \leq n-i \}.
$$
The layer of wall $\pyro{G_i}$ at height $k$  consists of $\sum_{m=0}^{k-1} s_{m}$ cubes of color 0 followed by $\sum_{m=k}^{n-i} s_{m}$ cubes of color 1.
The transformation $\phi: \pyro{G} \mapsto \pyro{H}$ maps $\pyro{G_i}$ to the $(n-i) \times i$ wall
$$
\pyro{H_{n-i}} = \{ (n-i,k,j) : 1 \leq k \leq n-i \mbox{ and } 1 \leq j \leq i  \}.
$$
We have inverted the colors and exchanged vertical and horizontal, so the tower of wall $\pyro{H_{n-i}}$ at $(n-i,k)$ consists of $\sum_{m=0}^{k-1} s_{m}$ cubes of color 1, stacked below  $\sum_{m=k}^{n-i} s_{m}$ cubes of color 0.

\begin{figure}[ht!]

\begin{center}

\begin{tikzpicture}

\begin{scope}[shift={(0,0)}]

\node at (0,0) {$
\begin{array}{cccccccc}
5 \\
3 & 5 \\
2 & 5 & 5 \\
1 & 2 & 2 & 3 \\
0 & 1 & 1 & 1 & 2 \\
0 & 0 & 1 & 1 & 2 & 2\\
0 & 0 & 1 & 1 & 1 & 1 & 1\\
\end{array}
$};

\node at (0,-2.5) {$\oo{G}$};

\draw (-1.85,-.35) -- (.85,-.35) -- (.85, -.8) -- (-1.85, -.8) -- cycle;

\draw[latex-latex] (1, -.35) -- (2, .35);

\end{scope}

\begin{scope}[shift={(11.5,0)}]

\node at (0,0) {$
\begin{array}{cccccccc}
2 \\
2 & 4 \\
1 & 4 & 5 \\
0 & 1 & 3 & 4 \\
0 & 0 & 1 & 1 & 1 \\
0 & 0 & 0 & 1 & 1 & 2\\
0 & 0 & 0 & 0 & 0 & 1 & 1\\
\end{array}
$};

\node at (0,-2.5) {$\oo{H}$};

\draw (-1.85,.8) -- (-.25,.8) -- (-.25, .35) -- (-1.85, .35) -- cycle;

\draw[latex-latex] (-2,.35) -- (-3, -.15);

\end{scope}

\begin{scope}[shift={(2.8,0)}, scale=.5]

\twocolor{1/4}{}

\begin{scope}[shift={(0,1)}]
\twocolor{4/1}{}
\end{scope}

\begin{scope}[shift={(0,2)}]
\twocolor{5/0}{}
\end{scope}

\draw[latex-latex] (2.5,-.5) to [bend right] (7,-1.5);

\node at (4.5, -2.5) {$\phi$};

\node at (.5,-1) {$\pyro{G_5}$};

\end{scope}

\begin{scope}[shift={(7,-1.5)}, scale=.5]

\twocolor{0/3}{}

\begin{scope}[shift={(0,1)}]
\twocolor{1/2}{}
\end{scope}

\begin{scope}[shift={(0,2)}]
\twocolor{1/2}{}
\end{scope}

\begin{scope}[shift={(0,3)}]
\twocolor{1/2}{}
\end{scope}

\begin{scope}[shift={(0,4)}]
\twocolor{2/1}{}
\end{scope}

\node at (.5,-1) {$\pyro{H_3}$};

\end{scope}

\end{tikzpicture}

\end{center}

\caption{An ogog triangle $\oo{G}$ from $\ogog{8}$ and its image $\oo{H}$ via the invertible mapping $\phi$. Row 5 of triangle $\oo{G}$ becomes pyramid wall $\pyro{G_5}$ which maps via $\phi$ to pyramid wall $\pyro{H_3}$ and then to row 3 of triangle  $\oo{H}$.}

\label{fig:gog-wall}

\end{figure}

We now translate the structure of pyramid $\pyro{H}$ into the triangle setting. Ogog triangle $\oo{H}$ has  $\oo{H}(n-i,k) = \sum_{m=0}^{k-1} s_{m}$ for $1 \leq h \leq n-i$, so its corresponding gog triangle $H$ has
\begin{equation}
\label{eqn:sum-of-gog}
H(n-i,k) = k + \sum_{m=0}^{k-1} s_{m} \quad \mbox{where} \quad 1 \leq k \leq n-i.
\end{equation}
The formulas in equations \eqref{eqn:sum-of-missing} and \eqref{eqn:sum-of-gog} are equivalent (taking $k = p+1$). Therefore, row $n-i$ of gog triangle $H$ contains the locations of the ones in the partial sum $\sum_{k=i+1}^n A_k$. In other words, gog triangle $H$ is constructed by considering the rows of alternating sign matrix $A$ in reverse order.
\end{proof}

\begin{proof}[Proof of Theorem \ref{thm:gog}]
Let $\gamma : \gog{n} \rightarrow \ogog{n} $ be the bijection $\gamma(G) = G - \mini{G}$.
Let $\pi : \ogog{n} \rightarrow \pyrogog{n}$ be the bjiection $\pi(\oo{G}) = \pyr{\oo{G}}$. By Corollary \ref{cor:gog}, the desired involution $f: \gog{n} \rightarrow \gog{n}$ is
$f =  \gamma^{-1} \circ \pi^{-1} \circ \phi \circ \pi \circ \gamma$.
\end{proof}

%% file: future.tex
\section{Conclusion and Future Work}

Poset refinements of the de Finetti Lattice $F_{n,2}$ have interesting combinatorical connections. We have shown that $\finetti{n}{2}$ is enumerated by the strict-sense ballot numbers  and that $\finetti{n}{2}{1}$ is enumerated by the ASM/TSSCPP sequence.
We have also shown that there is a very natural involution on gog triangles that corresponds to reversing the rows of the associated alternating sign matrices.
We conclude this work with some open research questions relating to both poset refinement and ASM/TSSCPP.

One natural continuation of this work is to consider the de Finetti refinements of the order ideal $B_{n,m}$ for $3 \leq m \leq n$, with $m=3$ as the obvious starting point. Analogous to Section \ref{sec:csp}, let $F_{n,m}$ to be the unique minimal de Finetti refinement of $B_{n,m}$. For $3 \leq m \leq n$, let $\finetti{n}{m}$ denote the collection of linear extensions of $F_{n,m}$ that adhere to de Finetti's condition (F2). For $1 \leq k \leq m$, let
$\finetti{n}{m}{k}$ denote the collection of minimal de Finetti refinements of $F_{n,m}$ such that every set of size at most $k$ is comparable with all other sets.
Are any of these families enunerated by known combinatorial sequences? If so, can we find a natural bijection to the appropriate combinatorial family? 

An understanding of these poset families could provide valuable insight into the family $\finetti{n}$ of de Finetti total orders. Any new perspective  could have ramifications for comparative probability orders and completely separable preferences. One could further investigate the subfamily of de Finetti refinements $\finetti{n}{m}{k}$ by defining a graph where we connect posets via an appropriately atomic operation, such as transpositions \cite{pruesse} for $\mathcal{L}(B_n)$  or flips \cite{maclagan} for members of $\finetti{n}$.
Also, can the dimension \cite{trotter} of a de Finetti refinement of $F_{n,m}$ be achieved by restricting ourselves to de Finetti refinements?

This paper brings  two  novel families into the fold of ASM and TSSCPP combinatorial structures: the poset refinements $\finetti{n}{2}{1}$ and the kagog triangles $\kagog{n}$. Some recent efforts have focussed on statistic-preserving bijections between subfamilies of ASM and TSSCPP structures \cite{biane,striker2018}. Perhaps the properties  $\finetti{n}{2}{1}$ and $\kagog{n}$ might reveal connections to help traverse the gap between ASM and TSSCPP. In particular, our two-color cube pyramid representation for triangular arrays revealed a natural bijection between magog triangles and kagog triangles, as well as a nice involution on gog triangles. We are optimistic that this point of view could aid in the investigation of the other known triangular families.